\documentclass[11pt]{article}
\usepackage{dsfont,amsmath,amsfonts,amssymb,amsthm,mathrsfs}
\usepackage{fullpage}
\usepackage{graphicx,graphics}
\usepackage{epsfig}
\usepackage{latexsym}
\usepackage[applemac,utf8]{inputenc}
\usepackage{ae,aecompl}
\usepackage[english]{babel}
\usepackage[colorlinks=true]{hyperref}
\usepackage[toc,page]{appendix}
\usepackage{bbm}
\usepackage[T1]{fontenc}
\usepackage[utf8]{inputenc}
\usepackage{stmaryrd}
\usepackage{comment}
\usepackage{pgf,tikz}
\usepackage{ulem} 
\date{}
\usepackage[letterpaper,top=2cm,bottom=2cm,left=2cm,right=2cm,marginparwidth=1.5cm]{geometry}
\usepackage[font=sf, labelfont={sf,bf}, margin=1cm]{caption}

\usepackage[auto]{contour}
\usetikzlibrary{shapes, angles, decorations.text, patterns, fadings,decorations.pathreplacing,fit}
\pgfdeclarelayer{edgelayer}
\pgfdeclarelayer{nodelayer}
\pgfsetlayers{edgelayer,nodelayer,main}

\newcommand{\xnun}{X^{(n)}_{\nu}}
\newcommand{\xmun}{X^{(n)}_{\mu}}
\newcommand{\Xmun}{\mathcal{X}^{(n)}_{\mu}}
\newcommand{\Xnun}{\mathcal{X}^{(n)}_{\nu}}
\newcommand{\dH}{\gamma}

\newcommand{\pp}{\mathsf{p}}
\newcommand{\cc}{\mathsf{c}}
\newcommand{\ff}{\cc}
\newcommand{\CRT}{\mathcal{T}}
\newcommand{\E}{\mathbb{E}}

\newcommand{\R}{{\mathbb{R}}}
\newcommand{\T}{\mathcal{T}}

\newcommand{\N}{\mathbb{N}}

\newcommand{\veps}{\varepsilon}

\newcommand{\pr}{\mathbb{P}}
\newcommand\numberthis{\addtocounter{equation}{1}\tag{\theequation}}
\renewcommand{\geq}{\geqslant}
\renewcommand{\leq}{\leqslant}

\newcommand{\tbr}{\mathcal{T}}

\def\llbracket{[\hspace{-.10em} [ }
\def\rrbracket{ ] \hspace{-.10em}]}
\newcommand{\dd}{\mathrm{d}}
\newcommand{\sT}{\mathsf{T}} 
\newcommand{\grow}{\operatorname{grow}}
\newcommand{\Cat}{\operatorname{Cat}}
\newcommand{\h}{\mathrm{ht}}

\renewcommand{\theequation}{\arabic{equation}}

\newtheorem{thm}{Theorem}[section]
\newtheorem{defn}[thm]{Definition}
\newtheorem{lem}[thm]{Lemma}
\newtheorem{prop}[thm]{Proposition}
\newtheorem{cor}[thm]{Corollary}

\newtheorem{rem}[thm]{Remark}

\date{}

\title{\bf Where do (random) trees grow leaves?}

\author{Alessandra Caraceni\thanks{Scuola Normale Superiore di Pisa, \url{alessandra.caraceni@sns.it}}, \hspace{0.2cm} Nicolas Curien\thanks{Universit\'e Paris-Saclay, \url{nicolas.curien@gmail.com}}\quad \& \hspace{0.2cm} Robin Stephenson\thanks{School of Mathematics and Statistics, University of Sheffield, \url{robin.stephenson@normalesup.org}}}

\begin{document}
\maketitle
\vspace{-1cm}
\abstract{We study a model of random binary trees grown ``by the leaves" in the style of Luczak and Winkler~\cite{LW04}. If $\tau_n$ is a uniform plane binary tree of size $n$, Luczak and Winkler, and later explicitly Caraceni and Stauffer, constructed a measure $\nu_{\tau_n}$ such that the tree obtained by adding a cherry on a leaf sampled according to $\nu_{\tau_n}$ is still uniformly distributed on the set of all plane binary trees with size $n+1$. It turns out that the measure $\nu_{\tau_n}$, which we call the leaf-growth measure, is noticeably different from the uniform measure on the leaves of the tree $\tau_n$. In fact we prove that, as $n \to \infty$, with high probability it is almost entirely supported by a subset of only
$$ n^{3 ( 2 - \sqrt{3})+o(1)} \approx n^{0.8038...} \mbox{ leaves}.$$
In the continuous setting, we construct the scaling limit of uniform binary trees equipped with this measure, which is the Brownian Continuum Random Tree equipped with a probability measure supported by a fractal set of dimension $ 6 (2 - \sqrt{3})$. We also compute the full (discrete) multifractal spectrum. This work is a first step towards understanding the diffusion limit of the discrete leaf-growth procedure. }

\begin{figure}[h!]
    \centering
    \includegraphics[width=.7\textwidth]{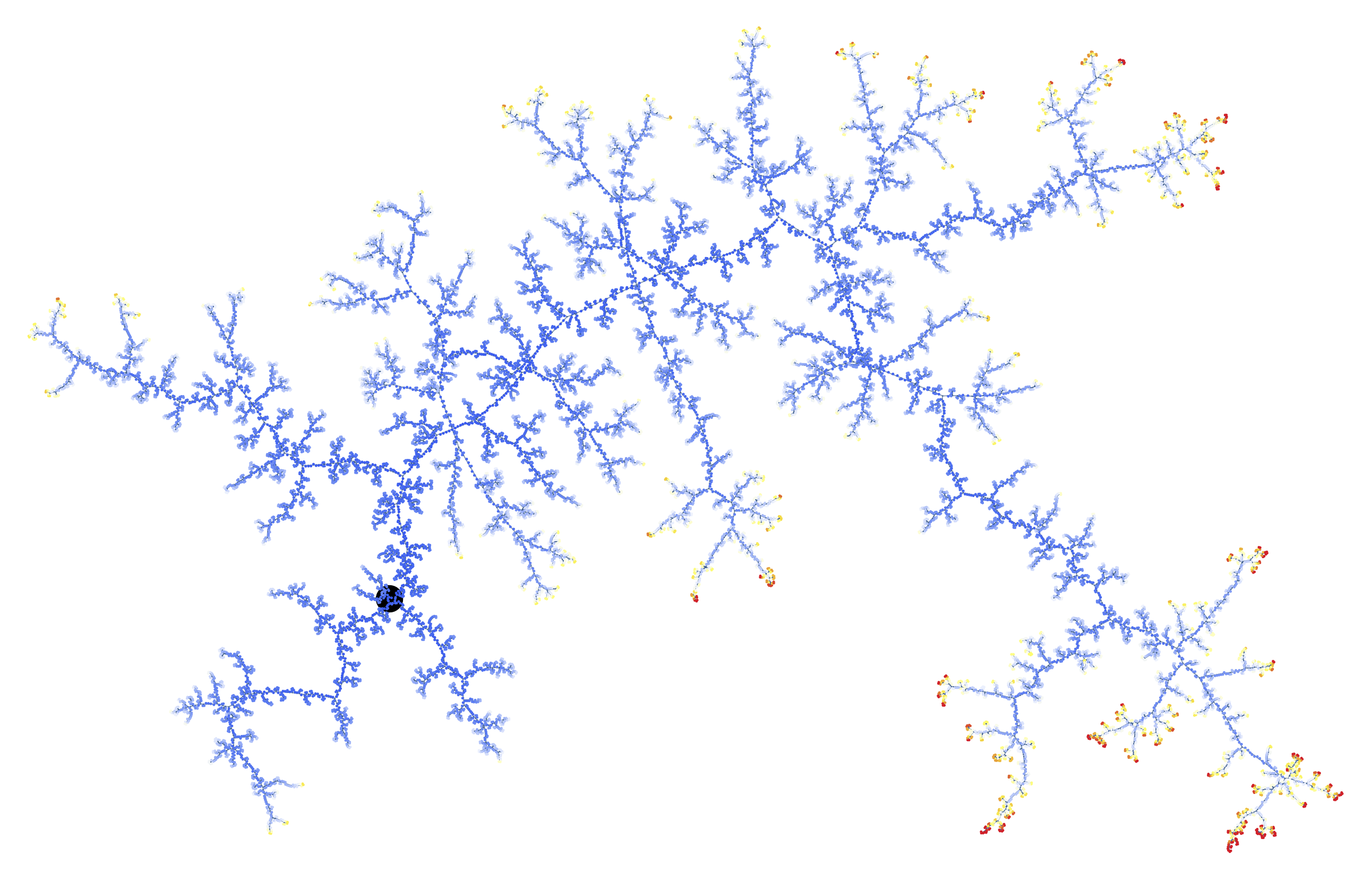}
    \caption{A uniform plane binary tree $\tau_n$ with 20 000 edges decorated with its leaf-growth measure $\nu_{\tau_n}$: the black disk represents the root vertex; and the color of the other  vertices (leaves) displays their $\nu_{\tau_n}$-mass (blue for small probability and red for high probability).
    }
    \label{fig:intro growing tree}
\end{figure}

\section{Introduction}

Within the very broad and far-reaching topic of randomly generated trees, models involving some sort of growth procedure arise very naturally and lead to a myriad of interesting questions and results. Well-known examples include random recursive trees and preferential attachment trees (e.g., the Barabasi--Albert model), as well as many tree models arising as data structures (binary search trees, $d$-ary trees, quadtrees, tries, etc.). We shall focus here on iterative growth procedures which, when performed up to size $n$, yield a tree that is uniformly distributed within the set of all size-$n$ plane $d$-ary trees. One famous such procedure is given by Rémy's algorithm and its variants~\cite{Rem85,Mar03,HS14}, but others include~\cite{marckert2023growing}.

In this paper, we shall consider a model of growth ``by the leaves'' in the style of what Luczak and Winkler introduced in~\cite{LW04}. Their question was the following: letting $\tau^{(d)}_n$ be a uniformly random plane $d$-ary tree with $n$ internal vertices (every vertex has either $d$ or $0$ children), is it possible to couple $\tau^{(d)}_n$ and $\tau^{(d)}_{n+1}$ in such a way that the latter is obtained from the former by adding $d$ children to a leaf of $\tau^{(d)}_n$ (see Figure~\ref{fig:leaf growth})? They answer the question in the affirmative, although their proof does not yield an explicit way of selecting a random leaf $l$, conditionally on $\tau^{(d)}_n$, in such a way that giving it $d$ new children will yield a uniform tree of size $n+1$. On top of its theoretical beauty, this result has had significant applications to stochastic domination of random trees, which was for example used to study planar maps \cite{addario2014growing, CS23}, the parking model on trees  \cite{goldschmidt2019parking}, and the number of spanning trees in the Erd\"os--R\'enyi random graph \cite{lyons2008growth}.

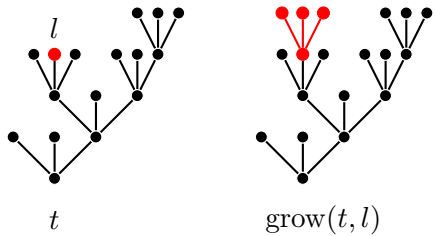
\begin{figure}[h!] \centering 
\begin{tikzpicture}[vertex/.style={circle, inner sep=1.6pt, draw=white, fill=black}, gem/.style={diamond, inner sep=1.7pt, fill=green!70!black}, branch/.style={very thick, brown},blossom/.style={green!70!black}, simple/.style={thick}, scale=1.1]
	\begin{pgfonlayer}{nodelayer}
		\node (0) at (5.75,3.25) {$t$};
		\node [style=vertex] (1) at (6.75, 4.75) {};

		\node [style=vertex] (3) at (8.25, 4.25) {};

		\node [style=vertex] (5) at (10, 5.25) {};

		\node [style=vertex] (7) at (6, 5.25) {};
		\node [style=vertex] (8) at (6.25, 4.25) {};
		\node [style=vertex, red, label=above:$l$] (9) at (5.75, 5.25) {};
		\node [style=vertex, red] (11) at (8.5, 5.75) {};
		\node [style=vertex] (12) at (10, 5.75) {};
		\node [style=vertex] (13) at (7, 5.75) {};
		\node [style=vertex] (14) at (7, 5.25) {};
		\node [style=vertex, red] (15) at (8.75, 5.25) {};
		\node [style=vertex] (16) at (9, 5.25) {};
		\node [style=vertex] (17) at (6.75, 5.75) {};
		\node [style=vertex] (21) at (10.25, 5.75) {};
		\node (23) at (9, 3.25) {$\grow(t,l)$};
		\node [style=vertex] (24) at (6.25, 4.75) {};
		\node [style=vertex] (25) at (6.75, 5.25) {};
		\node [style=vertex] (26) at (9.75, 4.75) {};

		\node [style=vertex] (28) at (6.5, 5.25) {};
		\node [style=vertex] (29) at (8.75, 4.75) {};
		\node [style=vertex] (30) at (5.75, 4.75) {};
		\node [style=vertex] (31) at (8.5, 5.25) {};

		\node [style=vertex] (33) at (5.75, 4.25) {};
		\node [style=vertex] (34) at (5.75, 3.75) {};
		\node [style=vertex, red] (35) at (8.75, 5.75) {};
		\node [style=vertex, red] (36) at (9, 5.75) {};
		\node [style=vertex] (37) at (9.25, 4.25) {};
		\node [style=vertex] (38) at (8.75, 4.25) {};
		\node [style=vertex] (39) at (9.75, 5.25) {};
		\node [style=vertex] (40) at (7.25, 5.75) {};
		\node [style=vertex] (41) at (9.25, 4.75) {};
		\node [style=vertex] (42) at (5.25, 4.25) {};
		\node [style=vertex] (43) at (9.5, 5.25) {};
		\node [style=vertex] (44) at (5.5, 5.25) {};
		\node [style=vertex] (45) at (9.75, 5.75) {};
		\node [style=vertex] (46) at (8.75, 3.75) {};

	\end{pgfonlayer}
	\begin{pgfonlayer}{edgelayer}
		\draw [style=simple] (29) to (37);
		\draw [style=simple] (43) to (26);
		\draw [style=simple] (45) to (5);
		\draw [style=simple] (31) to (29);
		\draw [style=simple] (44) to (30);
		\draw [style=simple, red] (11) to (15);
		\draw [style=simple] (13) to (14);
		\draw [style=simple] (41) to (37);
		\draw [style=simple] (8) to (34);
		\draw [style=simple] (30) to (8);
		\draw [style=simple] (5) to (26);
		\draw [style=simple, red] (35) to (15);
		\draw [style=simple] (38) to (46);
		\draw [style=simple] (28) to (1);
		\draw [style=simple] (1) to (8);
		\draw [style=simple] (25) to (1);
		\draw [style=simple] (16) to (29);
		\draw [style=simple] (24) to (8);
		\draw [style=simple] (3) to (46);
		\draw [style=simple] (15) to (29);
		\draw [style=simple] (17) to (14);
		\draw [style=simple] (21) to (5);

		\draw [style=simple] (12) to (5);

		\draw [style=simple] (40) to (14);
		\draw [style=simple] (37) to (46);
		\draw [style=simple] (7) to (30);

		\draw [style=simple] (42) to (34);
		\draw [style=simple] (26) to (37);

		\draw [style=simple] (9) to (30);

		\draw [style=simple] (14) to (1);
		\draw [style=simple] (33) to (34);
		\draw [style=simple] (39) to (26);
		\draw [style=simple, red] (36) to (15);
	\end{pgfonlayer}
\end{tikzpicture}
\caption{\label{fig:leaf growth}A plane ternary tree of size 5 (i.e., with 5 internal vertices and thus $3+2\cdot 4 =11$ leaves) is grown at the leaf $l$ by adding 3 children to $l$, thus creating the tree $\grow(t,l)$ of size 6.}
\end{figure}

In the case $d=2$ of plane binary trees, however, such a coupling between $\tau_n$ and $\tau_{n+1}$ -- where $d$ is henceforth omitted -- can be made very explicit (and in fact it is unique provided some natural symmetries are enforced). Given a uniform plane binary tree $\tau_n$ of \textbf{size} $n$ (i.e., with $n$ vertices that are \emph{not} leaves, and hence with $n+1$ leaves), one can produce an explicit measure $\nu_{\tau_n}$ on the set of the $n+1$ leaves of $\tau_n$ with the following property. If, conditionally on $\tau_n$, one samples $l$ according to $\nu_{\tau_n}$, then the tree $\operatorname{grow}(\tau_n,l)$ obtained by attaching a ``cherry" (two sister leaves) to $l$ is uniformly distributed among all plane binary trees with $n+1$ vertices. See Figure~\ref{fig:leaf growth}.

\paragraph{Construction of the leaf-growth measures.}

The (random) measure $\nu_{\tau_n}$, which we call the  \emph{leaf-growth} measure, is explicit as a function of $\tau_n$. Precisely, given $\tau_n=t$, and given any leaf $l$ of $t$, the mass given to $l$ is 
\begin{eqnarray} \label{eq:nunavecproduct}\nu_t(l)=\prod_{i=1}^{|l|}C(a_i,b_i), \end{eqnarray} 
where
\begin{itemize}
    \item $|l|$ is the height of $l$ in $t$;
    \item if $v_i$ is the ancestor of $l$ at height $i$, then $a_i$ is the size of the subtree of descendants of $v_i$ (including $v_i$ itself);
    \item $b_i=a_{i-1}-a_{i}-1$ (setting $a_0=n$) is the size of the subtree of descendants of the \emph{sibling} of $v_i$;
    \item $C(a,b)$ is the explicit rational function $$C(a,b) := \frac{(a+1)(2a+1)(a+3b+3)}{(a+b+1)(a+b+2)(2(a+b)+3)}.$$
\end{itemize}
See \cite{CS20} for details. In fact, there is a nice combinatorial interpretation to the expression $C(a,b)$ based on a ``best-of-three" match (as in plenty of competitive games and sports) which begs for a bijective explanation for the function $C$, see Remark \ref{rem:best of three}. 
The full details of the construction are given in Section~\ref{section:building schemes}.

\paragraph{Fractal properties of the leaf-growth measures.}
In this paper, we intend to investigate the nature of the measure $\nu_{\tau_n}$: how much does it differ from the uniform measure on the set of all leaves of $\tau_n$? In other words, how ``spread out'' do we expect it to be among all possible leaves?

We provide a rather detailed and explicit answer to this question. We shall prove that, with high probability, the measure $\nu_{\tau_n}$ is actually supported by a ``small'' set of leaves, whose size is of order $n^{3(2-\sqrt{3}) +o(1)}$. Indeed, in Section~\ref{section:discrete Hausdorff} we will show the following: 

\begin{thm}\label{thm:discretefractal}
Let $\tau_n$ be a uniformly distributed random binary tree of size $n$ and let $\nu_{\tau_n}$ be its leaf-growth measure; set 
$$\ \dH=3(2-\sqrt{3}).$$ 
For all $\epsilon>0$, we have 
$$\lim_{n\to\infty}\nu_{\tau_n}(\{l\in \tau_n \mid n^{-\dH-\epsilon}\leq \nu_{\tau_n}(l)\leq n^{-\dH+\epsilon}\})=1$$
in probability.

As a consequence, for all $\delta\in (0,1)$ there exists with high probability a set of leaves $A_{n,\delta}$ of $\tau_n$ such that $\nu_{\tau_n}(A_{n,\delta})\geq 1-\delta$ and $|A_{n,\delta}|\leq n^{\dH+\epsilon}$. Moreover, the maximal measure according to $\nu_{\tau_n}$ of a set of cardinality bounded above by $n^{\dH-\epsilon}$ tends to zero in probability as $n\to\infty$.
\end{thm}

\begin{figure}[h!]
    \centering
    \includegraphics[width=.7\textwidth]{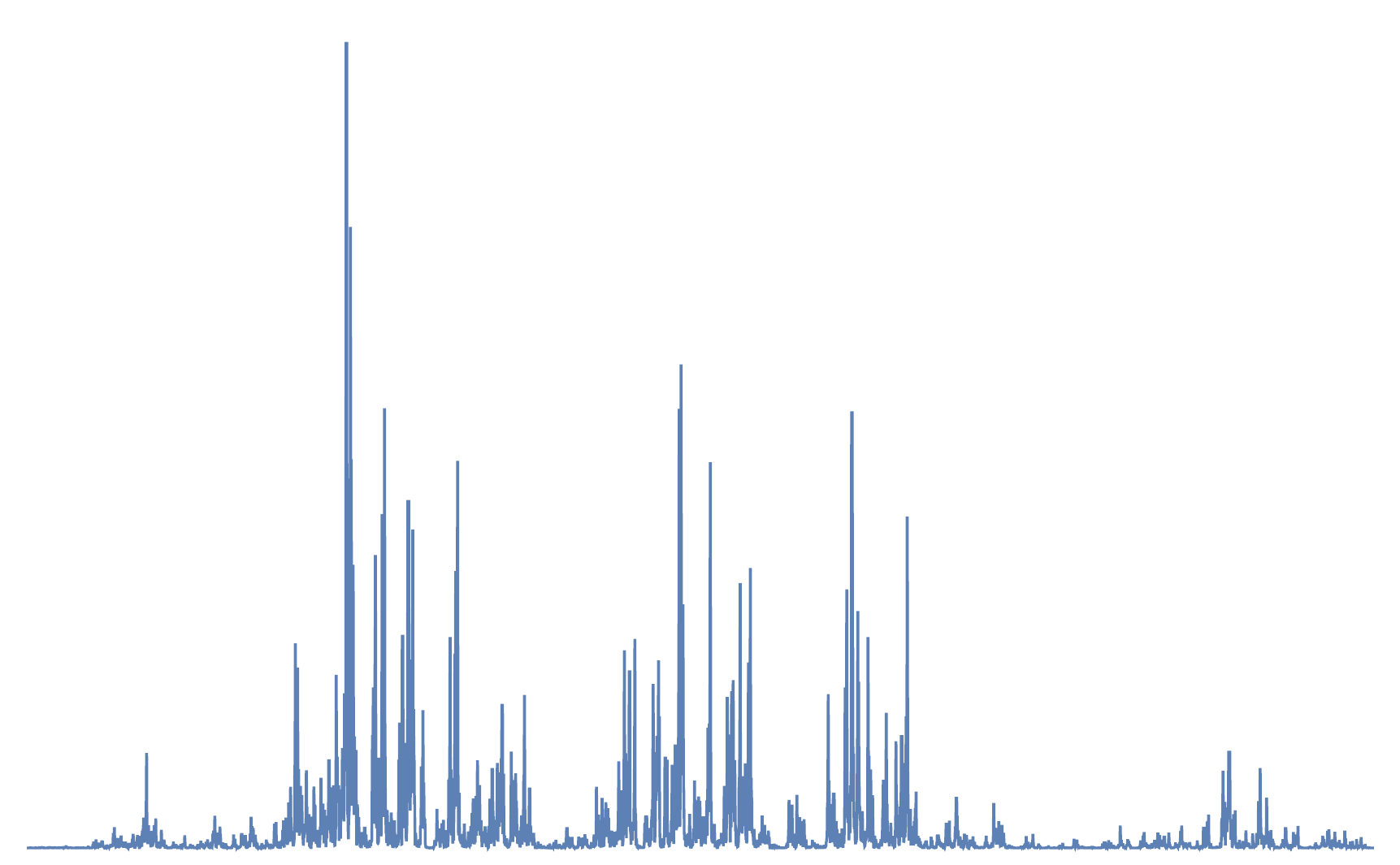}
    \caption{A plot of the density of the discrete leaf-growth measure with respect to the discrete uniform measure on the leaves of the tree depicted in Figure \ref{fig:intro growing tree} (starting from the root and going counterclockwise around the tree). In the limit, those two measures are mutually singular.
    }
    \label{fig:density}
\end{figure}

Refer to \cite[Theorem 1]{CLGharmonic} for a similar phenomenon for the discrete harmonic measure in random trees and \cite{lawler1993discrete} for a discrete analogue of Makarov's theorem. The above result is characteristic of a ``fractal'' behavior of the measure $\nu_{\tau_n}$ and indeed we shall see in Proposition~\ref{prop:continuous Hausdorff} that the continuous analogue of the measure $\nu_{\tau_n}$ is supported by a fractal subset of the Brownian CRT with dimension $ 2 \dH$. We actually go further and compute the full (discrete) multifractal spectrum of the measure $\nu_{\tau_n}$. In plain terms, although Theorem \ref{thm:discretefractal} shows that the typical $\nu_{\tau_n}$-mass of a leaf sampled according to $\nu_{\tau_n}$ is $n^{- \dH + o(1)}$, its higher moments are not ruled by the typical behavior (which is what we call multifractality):

\begin{thm}[Discrete multifractal spectrum]\label{thm:discrete multifractal spectrum} Given $\alpha\in\R$, let $\beta(\alpha)$ be the unique value of $\beta$ for which the integral
\begin{equation}
    I(\alpha,\beta)=\int_0^1 \dd x \frac{1}{\sqrt{x^3(1-x)^3}}\left(\ff(x)^{\alpha+1} x^{-\beta}+\ff(1-x)^{\alpha+1} (1-x)^{-\beta}-1\right)\label{eq:I(alpha,beta)}
\end{equation}
is zero, where $$\ff(x)=x^2(3-2x).$$
As $n \to \infty$, we have
$$   \mathbb{E}\left[\sum_{l \in \tau_n} \nu_{\tau_n}(l) ^{\alpha+1}\right] = n^{-\beta(\alpha) + o(1)}.$$
\end{thm}

\begin{figure}  
\begin{center}
\includegraphics[width=10cm]{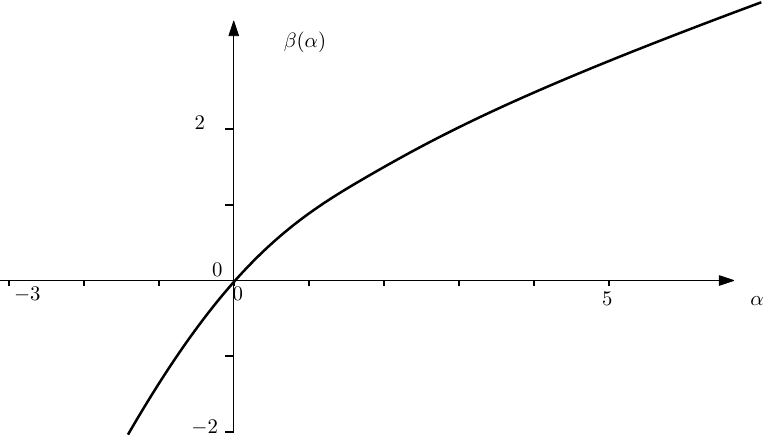}
\caption{Illustration of the function $ \alpha \to \beta( \alpha)$.}\label{fig:plotbeta}
\end{center}
\end{figure}
See Figure~\ref{fig:plotbeta}. In particular for $\alpha=k \in \{0,1,2,...\}$, the above display computes the (annealed) $k$th moment of the mass of a point sampled according to $\nu_{\tau_n}$ which is asymptotic to $n^{-\beta(k)+o(1)}$. In the case of the first moment $\alpha=1$ we have the explicit value $\beta(1) = \frac{1}{2}(5 - \sqrt{13})$. 
The proof of this result is presented in Section~\ref{section:multifractal spectrum}, along with some further comments on the minimal/maximal leaf-growth weight of a single leaf and the multi-fractal spectrum (see Remark~\ref{rem:multi-fractal}).

\paragraph{Continuum limit.}
The investigation of the natural discrete measure $\nu_{\tau_n}$ inevitably leads to the construction of a corresponding limiting measure $\nu_{\CRT}$ on the Brownian Continuum Random Tree (CRT) of David Aldous~\cite{Ald91a, Ald91}. In Section~\ref{section:convergence}, we show the following scaling limit result:

\begin{thm}\label{thm:ghpanddim} We have the following convergence in distribution for the Gromov--Hausdorff--Prokhorov topology (see Section \ref{sec:scalinglimitproof} for a definition):
$$\left( n^{-1/2} \cdot \tau_n, \nu_{\tau_n} \right) \xrightarrow[n\to \infty]{(d)}( 2\sqrt{2}\,\mathcal{T}, \nu_{ \mathcal{T}}), $$ where $ \mathcal{T}$ is the Brownian CRT  and $\nu_{\mathcal{T}}$ is a probability measure on the leaves of $\T$ with Hausdorff dimension $2\gamma= 6( 2- \sqrt{3}) \approx 1.6077$.
\end{thm}
Recall that the dimension of a measure $\nu$ on a metric space $E$ is the infimum of the Hausdorff dimensions of Borel subsets with full $\nu$-measure. 
Our main tool is a description of the law of the spine towards a $\nu_{\mathcal{T}}$-typical point which enables us to compute the Hausdorff dimension of the measure (Proposition \ref{prop:hd}), as well as distributional properties such as the (annealed) distribution of the height of a $\nu_{\mathcal{T}}$-typical point\footnote{This height turns out to have density $8x^3\mathrm{e}^{-2x^2}$ (see Proposition~\ref{prop:hd}): as one might expect, this height stochastically dominates that of a random leaf chosen \emph{uniformly}, since the leaf-growth measure must favour leaves belonging to larger subtrees. Unfortunately, we have no natural explanation for the surprisingly straightforward form of this density.}. Indeed, the continuous version of the measure sheds further light on many properties of its discrete counterparts, and techniques will be borrowed from the analysis of the continuous processes involved to produce the results already mentioned above.

\paragraph{Towards a diffusion on real trees.} The detailed analysis of the leaf-growth measure both in the discrete and continuous setting performed in this work can be seen as a first step in the ambitious programme of understanding the scaling limit of the leaf-growth process itself. Indeed, it is natural to postulate that the rescaled leaf-growth dynamics would converge towards an appropriate continuous Markov process with values in real trees whose invariant measure is the Brownian CRT (such a dynamic should be non-trivial by Proposition \ref{prop:pairedindep}).  Several such dynamics have been the object of intense study in recent years: from the Rémy dynamic, which almost surely converges in the GHP sense \cite{CurienHaas}, to the Aldous move on cladograms converging to the Aldous diffusion \cite{lohr2020aldous,FPRW18,FPRW20,forman2023aldous}; see also the literature about the process of root growth with regrafting \cite{MR2221786}. The  new continuous dynamic obtained by the leaf-growth process is constructed in the work \cite{CFT25} dealing with the much more general self-similar Markov trees, while  the discrete convergence will be addressed in a forthcoming work.

\medskip \noindent \textbf{Acknowledgments:}  The authors would like to thank Nic Freeman, Guillaume Conchon-Kerjan, Adrianus Twigt and Davide Lombardo for helpful discussions. The second author is supported by ``SuperGrandMa", the  ERC Consolidator Grant No 101087572. We are grateful to the two anonymous referees for their careful readings and their useful remarks that helped improve the paper.

\section{Building schemes for trees: leaf-growth measures}\label{section:building schemes}

In this section, we set up some notation and give some context for what we are about to discuss, i.e., how to ``uniformly grow'' trees ``from their leaves''.  For completeness, we shall state some known results in their general form, which involves $d$-ary trees (see~Luczak \& Winkler  \cite{LW04}). In the case $d=2$, motivated by applications to the study of mixing times of flip chains, this procedure has been made completely explicit by Caraceni and Stauffer \cite{CS20} and forms the basis of the current work.

\subsection{Building schemes for $d$-ary trees}

\begin{defn}\label{def:complete d-ary tree}
Let $d\geq 2$ and $n\geq 0$ be integers. We shall call a \emph{$d$-ary tree of size $n$} a rooted plane tree whose vertices all have $d$ children 
(we call such vertices \emph{internal}) or no children (we call such vertices \emph{leaves}).

We say that a $d$-ary tree $\tau$ has \emph{size} $|\tau|=n$ if it has $n$ internal vertices. For simplicity, we shall commit a slight abuse of notation and also write $\tau$ for the set of vertices of $\tau$, so we can write $v\in \tau$ to mean that $v$ is a vertex of $\tau$.

We call $\sT_n^{(d)}$ the set of all $d$-ary trees of size $n$ and write $\sT^{(d)}$ for $\cup_{n\geq 0}\sT_n^{(d)}$. Note that the set $\sT^{(d)}_0$ contains one element, the tree consisting of only its root.
\end{defn}

A $d$-ary tree has a natural recursive structure: if $t\in\sT_n^{(d)}$ with $n\geq 1$, it naturally induces a sequence of $d$-ary trees $(t^1,\ldots, t^d)$: letting $\rho\in t$ be the root, one can erase all children of $\rho$ other than the $i$th child $v_i$ and consider the connected component of $v_i$, rooted at $v_i$, to obtain $t^i$. One then has $\sum_{i=1}^d|t^i|=n-1$. From this decomposition one obtains the \emph{profile} of the tree $t$:

\begin{defn}
Given $t\in \sT_n^{(d)}$ with $n\geq 1$, its \emph{profile} is the vector $\underline{x}(t)=(|t^1|,\ldots,|t^d|)\in \mathbb{N}^d$.	
\end{defn}

Note that the term profile is often used with different meanings (for example, for the sequence of generation sizes) in similar contexts; however, we choose to employ it here to keep notation at least partially consistent with the original paper \cite{LW04}.

We now discuss a particular ``growth'' operation on $d$-ary trees. Given a pair $(t,l)$, where $t$ is a $d$-ary tree of size $n$ and $l$ is a leaf of $t$, we construct a tree $\grow(t,l)$ of size $n+1$ by adding $d$ leaves to $t$ as children of the vertex $l$, thus turning $l$ into an internal vertex. 

Luczak and Winkler consider the following natural question in \cite{LW04}: supposing $\tau_n$ is a uniform random element of $\sT^{(d)}_n$ and $\tau_{n+1}$ is a uniform random element of $\sT^{(d)}_{n+1}$, is it possible to couple $\tau_n$ and $\tau_{n+1}$ in such a way that $\tau_{n+1}$ is always of the form $\grow(\tau_n,l_n)$, where $l_n$ is some leaf of the tree $\tau_n$? 

They call a family of such couplings for all $n\geq 0$ a \emph{building scheme} for $d$-ary trees and answer the question in the affirmative, thus proving

\begin{thm}[Luczak, Winkler, 2004]\label{thm:building scheme existence}For all $d\geq 2$ there exists a building scheme for $d$-ary trees.
\end{thm}

Note that one can present a building scheme as a family of probability measures $\nu_t$, one for each $d$-ary tree $t$. The probability measure $\nu_t$ is defined on the set of leaves of $t$ and has the following property. If $(\tau_n,\ell)$ is a random variable such that $\tau_n$ is uniform in $\sT_n^{(d)}$ and, conditionally on $\tau_n=t$, $\ell$ is a leaf of $t$ distributed according to $\nu_t$, then $\grow(\tau_n, \ell)$ is uniform in $\sT^{(d)}_{n+1}$. 

Moreover, what Luczak and Winkler obtain is that there is a \emph{recursive} building scheme for $d$-ary trees, which we can think of as a family $(\nu_t)_{t\in \sT^{(d)}}$ such that for all $(t,l)$ with $t\in \sT^{(d)}$ of size $|t|\geq 1$, $l$ leaf of $t$, we have
$$\nu_t(l)=\sum_{i=1}^d \mathbf{1}_{l\in t^i}f_\nu(i;\underline{x}(t))\nu_{t^i}(l),$$
where $\underline{x}(t)$ is the profile of $t$ and we have omitted curly brackets from the arguments of the measure $\nu_t(\cdot)$ (as we shall do consistently from now on). The function $f_\nu(\cdot ;\cdot ):[d]\times \mathbb{N}^d\to [0,1]$, when computed for a certain $i\in[d]$ and the profile $\underline{x}(t)$ of a tree $t$, yields the probability $\nu_t(l\in t^i)$, which only depends on $\underline{x}(t)$ and not on the actual shape of the trees $(t^1,\ldots,t^d)$.

A recursive building scheme is completely determined by the functions $f_\nu(\cdot ; \cdot)$, since
$$\nu_t(l)=\prod_{v\prec l}f_\nu(i_v;\underline{x}^v),$$
where the product is taken over all internal ancestors $v$ of $l$, $v$ being the $i_v$th child of its parent and $\underline{x}^v$ being the profile of the tree of descendants of $v$, which serves as the root. Moreover, the functions $f_\nu$ must satisfy the obvious equation
$$\sum_{i=1}^df_\nu(i;\underline{x})=1,$$
as well as a recurrence ensuring they are a building scheme. Letting $\tau_n$ and $\tau_{n+1}$ be uniform in $\sT_n^{(d)}, \sT_{n+1}^{(d)}$ and letting $e_i$ be the $i$th canonical basis vector in $\mathbb{N}^d$, we must have
\begin{equation}\label{eq:d-ary recursion}\Pr((|\tau_{n+1}^i|)_{i=1}^d=\underline{x})=\sum_{i=1}^d\mathbf{1}_{\underline{x}_i\geq 1}\Pr((|\tau_{n}^i|)_{i=1}^d=\underline{x}-e_i)f_\nu(i;\underline{x}-e_i),\end{equation}
where $\underline{x}_i$ is the $i$th entry of the vector $\underline{x}$, which simply amounts to requiring that the building scheme act correctly on the probabilities of the profiles.

The proof of Theorem~\ref{thm:building scheme existence} given in \cite{LW04} leverages the fact that the probability that $\tau_n$ exhibits a certain profile can be computed explicitly in order to prove inequalities that are sufficient for the existence of a functions $f_\nu$ satisfying the above conditions and also taking values in $[0,1]$. This is thanks to a max-flow-min-cut type of argument, which does not suggest explicit expressions for possible solutions $f_\nu$, in spite of the fact that the coefficients of the recursions have fairly simple expressions.

It is at this point that we restrict ourselves to the simpler case $d=2$: we will now introduce some more compact notation and describe an explicit building scheme as done in \cite{CS20}.

\subsection{The leaf-growth measure for binary trees}

We shall henceforth only deal with what we will simply call \emph{binary trees}, that is, trees in $\sT^{(2)}$. We will drop $d$ from all notations, writing $\sT$ for $\sT^{(2)}$ and $\sT_n$ for $\sT_n^{(2)}$. Given $t\in \sT$ such that $|t|\geq 1$, we shall write $t^L$ and $t^R$ for the trees $t^1$ and $t^2$, identifying them as \emph{the left subtree} and \emph{the right subtree}. Moreover, we have the well-known result that binary trees are counted by Catalan numbers, i.e., that
\[|\sT_n|=\Cat(n)=\frac{1}{n+1}\binom{2n}{n}.\]
This implies that, given a profile $(a,b)\in\mathbb{N}^2$ and a uniform tree $\tau_{a+b+1}$ from $\sT_{a+b+1}$, we have
\[\Pr(\underline{x}(\tau_{a+b+1})=(a,b))=\frac{\Cat(a)\Cat(b)}{\Cat(a+b+1)}.\]
We shall need a compact notation for this expression, which from now on we will simply write as $P(a,b)$.

Now suppose we were to look for a recursive building scheme $f_\nu$ for binary trees, with the additional natural property of being \emph{symmetric}, that is, that $f_\nu(1;(a,b))=f_\nu(2; (b,a))$. In this case, $f_\nu$ can be computed explicitly and takes a very simple form. Indeed, as shown within the proof of \cite[Proposition 5.1]{CS20}, we have

\begin{thm}\label{thm:defmeasurediscrete}The function
\begin{equation}\label{eq:C(a,b)}C(a,b)=\frac{(a+1)(2a+1)(a+3b+3)}{(a+b+1)(a+b+2)(2(a+b)+3)},\end{equation}
defined on pairs $(a,b)$ of non-negative integers, is the only solution to the equations:
\[C(a,b)+C(b,a)=1;\]
\[P(a+1,b+1)=P(a,b+1)C(a,b+1)+P(a+1,b)(1-C(a+1,b))\]
for all integers $a,b\geq -1$, where we set $P(a,-1)=P(-1,b)=0$ (and therefore $C(a,b)$ appears in the equation only if $a,b\geq 0$).

Therefore, the only symmetric recursive building scheme for binary trees is the family of measures $(\nu_t)_{t\in \sT}$, where $\nu_t$ is the probability measure on the leaves of $t$ satisfying the recursion
\begin{equation}\nu_t(l)=1_{l\in t^L}C(|t^L|,|t^R|)\nu_{t^L}(l)+1_{l\in t^R}(1-C(|t^L|,|t^R|))\nu_{t^R}(l).\label{recursion with C's}\end{equation}
\end{thm}

From now on, we shall write $(\nu_t)_{t\in \sT}$ for the unique symmetric recursive building scheme for binary trees; we will call $\nu_t$ the \emph{leaf-growth measure} on $t$. Notice that the equations in the statement above are precisely the requirement \eqref{eq:d-ary recursion} in order for $\nu_t$ to be a building scheme, with the added requirement of symmetry.

Remark that the uniqueness of the solution is clear, since one can inductively compute $C(a,b)$, for $a>b$, by using equations of the third type to reduce the value of $b$ until it becomes $0$ (while obviously $C(a,a)=\frac12$). The fact that the expression \eqref{eq:C(a,b)} is a solution can be easily checked by induction.


\begin{rem}[A sports question] \label{rem:best of three}
The expression \eqref{eq:C(a,b)} for $C(a,b)$ has an interesting combinatorial interpretation. Imagine a match played between $L$ (``left'') and $R$ (``right''), with $L$ initially having $2a+1$ tokens and $R$ having $2b+1$. Each player has probability of winning the match proportional to their number of tokens: for the first match, $L$ wins with probability $\frac{2a+1}{2a+2b+2}$ and $R$ wins with probability $\frac{2b+1}{2a+2b+2}$. After a match is played, the winner gains one token. Any subsequent matches have updated outcome probabilities, but are played independently.

The quantity $C(a,b)$ is precisely the probability that $L$ wins a majority of three successive matches, that is,
$$C(a,b)=3\frac{(2a+1)(2a+2)(2b+1)}{(2a+2b+2)(2a+2b+3)(2a+2b+4)}+\frac{(2a+1)(2a+2)(2a+3)}{(2a+2b+2)(2a+2b+3)(2a+2b+4)}.$$

We do not have a direct combinatorial proof of the link between the tree growth measure and this presentation in terms of best-of-three winning probabilities, but this interpretation can be used to give alternative descriptions of the recursive symmetric building scheme, such as the one that follows.
\end{rem}

\begin{cor}
Given $n\geq 1$ and $t\in\sT_n$, place $n+1$ tokens numbered $1$ to $n+1$ on the leaves of $t$. Now play the following game:
\begin{itemize}
\item at each step, pick a vertex $v$ with tokens such that its descendants have no tokens (initially, one can pick any leaf of $t$); move all the tokens of $v$ to the parent of $v$;
\item if a vertex $w$ contains $n_a>0$ tokens numbered $a$ and $n_b>0$ tokens numbered $b\neq a$, play a best-of-three game on the vertex according to the rules described in Remark \ref{rem:best of three} with one player having $n_a$ starting tokens and the other having $n_b$. After all three matches are played, three tokens have been added and the winner has been determined. Relabel all tokens at $w$ with $a$ or $b$, according to which player was the winner, and destroy two tokens (thus leaving $n_a+n_b+1$ on $w$).
\item unless all tokens have the same label, perform another step.
\end{itemize}

The leaf corresponding to the final surviving label is distributed according to $\nu_t$.
\end{cor}

\begin{proof}We say that a pile of tokens on a vertex $v$ is \emph{movable} if they all have the same label and there are no tokens on any strict descendants of $v$. Letting $t^v$ be the tree of descendants of $v$, if the tokens of $v$ are movable, then they are $2|t^v|+1$. Single tokens on leaves are movable, and they are indeed $2\cdot 0+1$. Inductively, when a new movable pile is made on an internal vertex $v$, one can say that it must have been created by merging a movable pile from its left child $v_l$ and one from its right child $v_r$, which must be of sizes $2|t^{v_l}|+1$ and $2|t^{v_r}|+1$, respectively. Thus a best-of-three game is played in $v$, yielding a movable pile of $2|t^{v_l}|+2|t^{v_r}|+2+3-2=2|t^v|+1$ tokens. The probability that the label of the tokens is that of a leaf of $t^{v_l}$ is $C(|t^{v_l}|,|t^{v_r}|)$.
The procedure as described plays a best-of-three game in each internal vertex of $t$. Letting $t^L, t^R$ be the left and right subtrees, the probability that the final label is $l\in[n+1]$ is indeed (inductively) given by~\eqref{recursion with C's}.
\end{proof}

\subsection{Some useful asymptotics}
Before we move on to investigate several aspects of the leaf-growth measures $\nu_t$, it is convenient to state some of the main (easy or well known) approximation results we will need in what follows, both for the quantities $C(a,b)$ and the probabilities $P(a,b)$. \medskip

Note that, fixing $b=n-1-a$, we can write out the values of $P(a,b)$ as a quotient of factorials
\[P(a,b)=\frac{(2a)!(2b)!n!(n+1)!}{a!(a+1)!b!(b+1)!(2n)!}.\]
Using Stirling's approximation, it is clear that, for $x\in(0,1),$
\begin{equation}\label{Stirlinglimit}
4\sqrt{\pi}\lim_{\substack{n\to\infty \\ a/n \to x}}n^{3/2} P(a,b)=\frac{1}{x^{3/2}(1-x)^{3/2}}=:\pp(x);
\end{equation}
moreover, for all $x\in [0,1]$,
\begin{equation}\label{eq:cc}
 \lim_{\substack{n\to\infty \\ a/n \to x}}C(a,b)=x^2(3-2x)=:\cc(x).\end{equation}

The more compact notation $\pp(x)$ and $\cc(x)$ for the functions $x^{-3/2}(1-x)^{-3/2}$ and $x^2(3-2x)$ (see Figure~\ref{fig:graph of c}) that is being introduced here will be used frequently in future sections.
 \begin{figure}[h!]
\begin{center}
    \includegraphics[width=8cm]{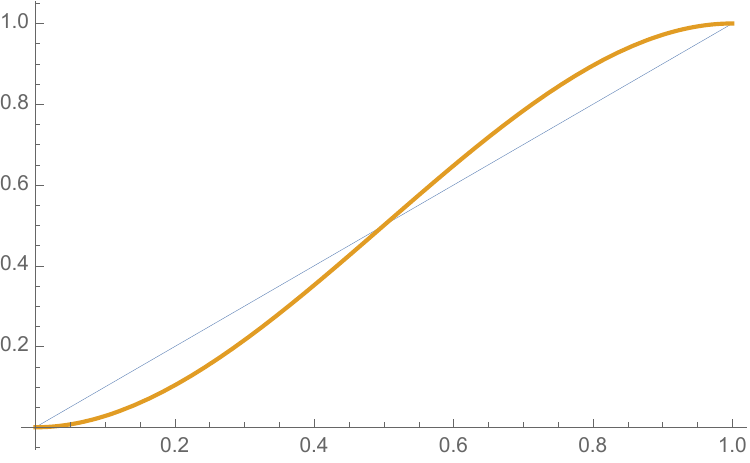}
    \caption{\label{fig:graph of c}A plot of the function $ x \mapsto \ff(x)$. Notice that it is symmetric with respect to $(1/2,1/2)$ and below the first bisector when $x \in [0,1/2]$.}
    \end{center}
\end{figure}

\vspace{0.5cm}

Over the course of the paper we will encounter several Riemann sums (of improper integrals) of the form
$
\sum_{a=1}^{n-1} f(\frac{a}{n})P(a,n-1-a) C(a,n-1-a)^k
$ for some integer $k \in \{0,1,2,... \}$. Under some mild conditions on $f$ and $k$, they exhibit the expected behaviour:

\begin{prop}\label{prop:generalintegrallimit}

Let $f$ be a continuous function on $(0,1),$ and $k\in \{0,1,2, ... \}.$ Assume that $|f(x)|\leq  K(x \wedge (1-x))^{1-2k}$ for some $K>0,$ and if $k>0$, assume further that $\sup\{f(\frac{a}{n}),a\in\{1,\ldots,n-1\}\}=o(\sqrt{n}).$ We then have

\begin{equation}\label{eq:limsumwithC}
    \underset{n\to\infty}\lim\sqrt{n} \sum_{a=1}^{n-1} f{\left(\frac{a}{n}\right)}P(a,n-1-a)\,C(a,n-a-1)^k = \frac{1}{4\sqrt{\pi}}\int_0^1 f(x)\ff(x)^k \pp(x) \dd x.
\end{equation}

\end{prop}

Note that, while the $k=0$ case is essentially a variant of \cite[Proposition 39]{HM12}, we prefer giving a self-contained proof here.
\begin{proof} The following uniform bounds on $C$ and $P$ are easily checked: there exists a constant $K'>0$ such that
\begin{equation}\label{Stirlingbound}
n^{3/2} P(a,b) \leq \frac{K'}{(\frac{a}{n})^{3/2}(\frac{b}{n})^{3/2}},
\end{equation}
and  
\begin{equation}\label{eq:boundc1}
\left|C(a,b)-\cc{\left(\frac{a}{n}\right)}\right|\leq\frac1n,
\end{equation}
where $n=a+b+1$, and \eqref{Stirlingbound} is immediately obtained by Stirling's approximation.

We start with the $k=0$ case. Specifically, we are going to show
\[
    \sqrt{n} \sum_{a=2}^{\lfloor n/2\rfloor} f{\left(\frac{a}{n}\right)}P(a,n-1-a) \longrightarrow \frac{1}{4\sqrt{\pi}}\int_0^{1/2} f(x) \pp(x) \dd x.
\]
This is enough to conclude by first exploiting the natural symmetry around $1/2$, then observing that $P(\lfloor n/2\rfloor,\lfloor n/2 \rfloor )=O(n^{-3/2})$ so overcounting middle terms is not an issue, and finally noting that $f(1/n)P(1,n-2)=O(1/n),$ so removing the first term is also not an issue.

Write $n^{1/2}\sum_{a=2}^{\lfloor n/2\rfloor} f(\frac{a}{n})P(a,n-1-a)$ as $\int_{0}^{1/2} g_n(x)\mathrm{d}x$, where $g_n$ is the function defined by $g_n(x)=n^{3/2}f(a/n)P(a,n-1-a)$ for $x\in \left[\frac{a-1}{n},\frac{a}{n}\right)$ and $a\in \{2,\ldots,\lfloor n/2\rfloor\},$ and $g_n(x)=0$ otherwise. It is readily checked that $g_n(x)$ converges pointwise to $\frac1{4\sqrt{\pi}}f(x) \pp(x),$ and we just need to check that we can apply the dominated convergence theorem. To do this, notice that, if $x\in \left[\frac{a-1}{n},\frac{a}{n}\right)$ for $2\leq a \leq \lfloor n/2\rfloor$, then $x+1/n\leq 2x,$ while $n-1-a\geq n/3$. Hence, by \eqref{Stirlingbound}, we have
\begin{align*}
    g_n(x)& =n^{3/2}f(a/n)P(a,n-1-a) \\
    &\leq KK'\frac{a}{n} \left(\frac{a}{n}\right)^{-3/2}\left(\frac{n-1-a}{n}\right)^{-3/2} \\
    &\leq 3^{-3/2}KK'   \left(x+\frac{1}{n}\right)x^{-3/2} \\
    &\leq 2\,3^{-3/2}KK'x^{-1/2},
\end{align*}
which ends the proof for $k=0$.

\vspace{1cm}

Next we treat the $k>0$ case. Notice first that, since $|C(a,n-1-a)-\ff(\frac{a}{n})|\leq \frac{1}{n}$ and both terms are less than one, we have $|C(a,n-1-a)^k-\ff(\frac{a}{n})^k|\leq \frac{k}{n}$. Hence we can write
\begin{align*}
   &\sqrt{n} \sum_{a=1}^{n-1} f{\left(\frac{a}{n}\right)}P(a,n-1-a)\,C(a,n-a-1)^k \\
   =&\sqrt{n}\sum_{a=1}^{n-1} f{\left(\frac{a}{n}\right)}P(a,n-1-a)\,\ff{\left(\frac{a}{n}\right)}^k +O\left( \frac1{\sqrt{n}}\sum_{a=1}^n f{\left(\frac{a}{n}\right)}P(a,n-1-a)\right) \\
   =&\sqrt{n}\sum_{a=1}^{n-1} f{\left(\frac{a}{n}\right)}P(a,n-1-a)\,\ff{\left(\frac{a}{n}\right)}^k+ O\left(\frac1{\sqrt{n}}\sup\left\{f{\left(\frac{a}{n}\right)},a\in\{1,\ldots,n-1\}\right\}\sum_{a=1}^{n-1}P(n,n-1-a)\right)
 \end{align*}
By assumption, the second term tends to $0$. We treat the first term by applying the $k=0$ case to the function $f(\cdot)\ff(\cdot)^k$, which does satisfy the desired condition that $|f(x)\ff(x)^k| \leq K(x\wedge (1-x))$ for some $K>0$, because $|f(x)|\leq \frac13K(x\wedge(1-x))^{1-2k}$ for well-chosen $K$ and $\ff(x) \leq 3 (x\wedge (1-x))^2$.
\end{proof}

\section{The leaf-growth measure on the Brownian CRT}\label{section:convergence}

In this section we construct the leaf-growth measure $\nu_\mathcal{T}$ on the Brownian CRT and prove that (a scalar multiple of) the latter, equipped with the former, is the scaling limit of the discrete model for the Gromov--Hausdorff--Prokhorov metric. In order to establish the scaling limit, we provide an explicit description of the spine towards a typical $\nu_\mathcal{T}$-point.

\subsection{Preliminary notation}

The Brownian Continuum Random Tree (CRT) is the random real tree $( \mathcal{T}, d)$ coded by a standard Brownian excursion. We refer to \cite{LG2005} for its main properties and for standard notation about real trees, some of which is recalled below. It is given with a distinguished point $ \rho$ called the  \emph{root} of the tree. The degree of a point $x \in \mathcal{T}$ is the number of connected components of $\T\backslash \{x\}$. When the degree is equal to $1$ we say $x$ is a \emph{leaf}; we call $x$ a \emph{point of the skeleton} when its degree is $2$ and a \emph{branchpoint} when its degree is $3$ (there are almost surely no points with degree $4$ or more in the Brownian CRT). For $x\in \T,$ we call \emph{height} of $x$ its distance to $\rho$ and denote it by $\h(x)$. We say that $x$ is a \emph{descendant} of another point $y\in\T$ (and equivalently that $y$ is an \emph{ancestor} of $x$) if $y$ lies on the segment between $\rho$ and $x$, and we call $\T_x$ the subset of $\T$ formed by $x$ and all its descendants. Finally, for $x \in \T$ and $t\geq 0$ we define for the point $[x]_t$  as:
\begin{itemize}
    \item the unique ancestor of $x$ with height $t$, if $t< \h(x)$
    \item $x$ itself, if $t\geq \h(x).$
\end{itemize}
In particular, $ \T_{[x]_t}$ denotes the subtree above height $t$ which contains the point $x$. When $[x]_t$ is a branchpoint and if $t < \h(x)$, then $\T_{[x]_t} \backslash \{[x]_t\}$ has two components and we denote by $ \mathcal{T}_{[x]_{t+}}$ the closure of component that contains~$x$.

The Brownian CRT is also endowed with its uniform mass measure $\mu$ (also called the leaf measure), which can either be seen as the projection of the Lebesgue measure on $[0,1]$ in the coding by the Brownian excursion, or more intrinsically as its natural Hausdorff measure, see \cite{duquesne2005probabilistic,duquesne2006hausdorff}. 

In our work, another description of $\mathcal{T}$ will be of key importance. Namely, the Brownian CRT can (with this normalisation) be seen as a fragmentation tree in the sense of Haas and Miermont \cite{HM04} with self-similarity index $\alpha = - \frac{1}{2}$ for the dislocation measure 
\begin{eqnarray}
    \label{eq:Dislocation}
 \mathsf{Dis} := \sqrt{\frac{2}{\pi}} \int_{1/2}^{1} \frac{ \mathrm{d}x}{(x(1-x))^{3/2}} \delta_{(x,(1-x))}.
 \end{eqnarray}
Heuristically, this means that it encodes the genealogy of a system of particles starting from a single particle of mass $1$, where particles of mass $m$ evolve independently and split into two particles of mass $mx$ and $m(1-x)$ at rate $m^{-1/2} \cdot \sqrt{\frac{2}{\pi}} \pp(x)$, where $\pp(x)$ is as in \eqref{Stirlinglimit}. Following Bertoin \cite{bertoin2002self}, those particles alive at time $t$ correspond to the subtrees of $\T$ above height $t$ and their masses correspond to the $\mu$-masses of those subtrees.

\subsection{Definition of $\nu$}
Conditionally on the CRT~$\mathcal{T}$, the leaf-growth measure $\nu_{\mathcal{T}}$, which throughout this section we will also denote by $\nu$, rendering the dependence on $\mathcal{T}$ implicit, will be defined by specifying the value $\nu(\T_x)$ for all non-leaf $x\in\T$. Proposition 1 of \cite{Steph13} will then guarantee its existence and uniqueness. The definition will be analogous to the discrete case: there, the relative masses given to the subtrees in an $(a,b)$ split were $C(a,b)$ and $C(b,a),$ and here the relative masses given to the subtrees in a $(x,1-x)$ split will be $\ff(x)$ and $1-\ff(x)=\ff(1-x).$ 

Note that the construction is purely deterministic: for the purposes of Proposition~\ref{prop:def:nu}, one can take $\T$ to be any binary compact rooted real tree equipped with a measure $\mu$ supported by the leaves, and $ \mathsf{c}$ to be any symmetric function satisfying $ \mathsf{c}(x) \leq x$ for $x \in [0,1/2]$. Since our only application concerns the Brownian CRT with the specific function $\mathsf{c}=x^2(3-2x)$, for the sake of simplicity we shall keep using the the notations $\mathcal{T}$, $\mu$ and $\ff$ throughout this section, though they should be interpreted here with this more general framework in mind.

\medskip 

Given $x \in \T$, recall from above that we denote by $ \mathcal{T}_{[x]_t}$ and $\mathcal{T}_{[x]_{t+}}$, respectively, the subtree above height $t$ which contains $x$, and in the case of a branchpoint, the closure of the connected component of $\T\setminus{[x]_t}$ which contains $x$. Finally we set 
$$ u_x(t) := \frac{\mu\big( \T_{[x]_{t+}}\big)}{ \mu \big( \T_{[x]_t}\big)} \in [0,1].$$
In particular, since the tree $ \mathcal{T}$ is binary and compact, there are for any $x \in \T$ at most countably many branchpoints on the ancestral path of $x$, hence we have $u_x(t) < 1$ only for a countable number of heights $t \leq \h(x)$. In particular, if $(t_i : i \in  \mathbb{N})$ is an enumeration of the heights of those branchpoints then the family $ (\mu( \mathcal{T}_{[x]_{t_i}})-\mu( \mathcal{T}_{[x]_{t_i +}}): i \in \mathbb{N})$ is summable, and since the measure $\mu$ does not charge the skeleton, its sum is equal to $1- \mu( \mathcal{T}_x)$. We deduce then that for all $x \in \T$ we have
 \begin{eqnarray} \mu\big( \T_x\big) = \prod_{0 \leq t < \h(x)} u_x(t). \label{eq:noloss}\end{eqnarray}

\begin{prop}[Definition of $\nu$]\label{prop:def:nu} Suppose that $ \mathcal{T}$ is a compact binary rooted tree equipped with a measure $\mu$ supported by its leaves. If $\mathsf{c}:[0,1] \to [0,1]$ is symmetric with respect to $1/2$ and satisfies $ \mathsf{c}(x) \leq x$ for $x \in [0,1/2]$, then there exists a unique Borel measure $\nu$ supported by the leaves of $ \mathcal{T}$ specified by 
\begin{eqnarray} \nu( \T_x) := \prod_{0 \leq t < \h(x)} \ff( u_x(t)). \end{eqnarray}
\end{prop}
\begin{proof} Let, for $x\in\T,$ $m(x)=\prod_{0 \leq t < \h(x)} \ff( u_x(t)).$ Since we have $ \mathsf{c}(x) \leq x$ for $x \in [0,1/2]$ the product is well-defined and is non-zero as soon as $\prod_{0 \leq t < \h(x)}  u_x(t)$ is non-zero. One can straightforwardly check the following:
\begin{itemize}
    \item $m$ is a well-defined and nonnegative function, which is equal to $0$ exactly on the leaves of $ \mathcal{T}.$
    \item $m$ is nonincreasing, in the sense that, if $x$ is an ancestor of $y$, then $m(x)\geq m(y).$
    \item $m$ is ``left-continuous", in the sense that $\displaystyle m(x)=\lim_{s\to (\mathrm{ht}(x))^{-}} m([x]_s).$
    \item $m$ is ``additively right-continuous": if $x$ is not a branchpoint, then $m$ is continuous at $x$, while if $x$ is a branchpoint, then, letting $y$ and $z$ be two descendants of $x$ not in the same one of the two subtrees originating at $x$, then $m(x)=\displaystyle \lim_{s\to (\mathrm{ht}(x))^+} m([y]_s)+m([z]_s).$
\end{itemize}
These facts, combined, let us use \cite[Proposition 1]{Steph13}, which then tells us there exists a unique measure $\nu$ such that $\nu(\T_a)=m(a)$ for $a$, and that it has no atoms. With the notation introduced just before the Proposition, since $\mu$ does not charge the skeleton, we can also check that $1-\nu( \mathcal{T}_x) = \sum_{i \in \mathbb{N}} \big(\nu ( \mathcal{T}_{[x]_{t_i}}) -\nu ( \mathcal{T}_{[x]_{t_i +}})\big)$, in particular $\nu$ gives no mass to the ancestral path of $x$, and as a result is supported by the leaves of $ \mathcal{T}$.
\end{proof}

\subsection{Path towards a typical $\nu$-leaf and spine decomposition}
Recalling the definition of the Brownian dislocation measure in \eqref{eq:Dislocation}, we now build two measures $\pi_\mu$ and $\pi_\nu$ on $(0,\infty)$ defined by
\begin{align*}
\int_{0}^{\infty} \mathrm{d} \pi_\mu(s) \  \phi(s) &:= \sqrt{\frac{2}{\pi}} \int_{1/2}^{1} \left(\ff(x) \phi(-\log (x)) + \ff(1-x) \phi(-\log (1-x))\right)\pp(x)\dd x \\
&= \sqrt{\frac{2}{\pi}} \int_0^{1}  \ff(x) \phi\left(-\log (x)\right) \pp(x)\dd x
\end{align*} 
and
\begin{align*}  \int_{0}^{\infty} \mathrm{d} \pi_\nu(s) \  \phi(s) &:= \sqrt{\frac{2}{\pi}} \int_{1/2}^{1} \left(\ff(x) \phi(-\log (\ff(x))) + \ff(1-x) \phi(-\log (\ff(1-x)))\right)\pp(x)\dd x \\
&= \sqrt{\frac{2}{\pi}} \int_0^{1}  \ff(x) \phi\left(-\log ( \ff(x))\right)\pp(x)\dd x.
\end{align*} 
Consider then $\Pi_\mu$  a Poisson point process (P.p.p.) on $ \mathbb{R}_+ \times \mathbb{R}_+$ with intensity $ \mathrm{d} \pi_\mu \otimes \mathrm{d}t$. By replacing each atom $(p,t)$ of $\Pi_\mu$ by the atom $(-\log(\ff ( \mathrm{e}^{-p})),t)$ we obtain another Poisson point process with intensity $ \mathrm{d}\pi_\nu \otimes \mathrm{d}t$. We can then define two coupled (non-decreasing) subordinators $\xi_\mu$ and $\xi_\nu$  started from $0$, and with respective L\'evy measures $\pi_\mu$ and $\pi_\nu$, and obtained by summing the atoms of the respective P.p.p., i.e.,
\begin{equation}\label{defxi}
 \xi_\mu(t) = \sum_{\begin{subarray}{c}(p,s) \in \Pi_\mu \\ s \leq t\end{subarray}} p \quad \mbox{and} \quad \xi_\nu(t) = \sum_{\begin{subarray}{c}(p,s) \in \Pi_\nu \\ s \leq t \end{subarray}} p = \sum_{\begin{subarray}{c}(p,s) \in \Pi_\mu \\ s \leq t\end{subarray}} -\log(\ff ( \mathrm{e}^{-p})).\end{equation} 
Then we perform the Lamperti transformation with index $1/2$ using $\xi_\mu:$ we let
\begin{equation}\label{eq:timechange}\theta(t)=\inf\left\{u\geq 0: \int_0^u  \mathrm{e}^{-\frac12 \xi_{\mu}(r)}\mathrm d r > t\right\},\end{equation}
and build two processes $X_\mu$ (a so-called positive self-similar Markov process, see \cite{Lamp62}) and $X_\nu$ (which is not Markovian, since the time change uses information from $X_\mu$) defined for $t\geq0$ by
\begin{eqnarray}
(X_{\mu}(t),X_\nu(t)):=(\mathrm{e}^{-\xi_\mu(\theta(t))},\mathrm{e}^{-\xi_\nu(\theta(t))}). \label{eq:lamperti}\end{eqnarray}
Both of those processes are pure jump and non-increasing from $1$ to $0$, and are absorbed at $0$ at time \begin{eqnarray} \label{def:I} I:=\int_{0}^{\infty}\mathrm{e}^{-\frac12 \xi_\mu(t)}\mathrm{d}t\end{eqnarray} which is almost surely finite since $ \xi_\mu$ drifts to $+\infty$ at linear speed.

\begin{thm}[Spine decomposition]\label{thm:spinedec} Conditionally on the Brownian CRT $ \mathcal{T}$, let $L$ be a $\nu$-distributed random point of $\tbr.$ Then the bivariate process $$(\mu(\T_{[L]_{t+}}),\nu(\T_{[L]_{t+}}))_{t\geq 0}$$ (under the annealed law) is distributed as $(X_{\mu},X_{\nu}).$
As a consequence, $L$ is almost surely a leaf, and its height has the law of $I$.

Moreover, let $(t_i,i\in\N)$ be the times at which $u_L(\cdot) \in (0,1)$, and for all $i\in \N$ let $\T_i$ be the subtrees which branch off $\llbracket \rho,L\rrbracket$ at those times (ranked, say in decreasing order of their $\mu$-masses). Then, conditionally on the whole process $(\mu(\T_{[L]_{t}}),t\geq 0),$ the trees $$\big(u_L(t_i) \cdot \mu(\T_{L(t_i)}) \big)^{-1/2} \cdot \T_i, \quad i \geq 1$$ are independent Brownian CRT's.
\end{thm}

\begin{proof}
Informally, the idea is that the path to $L$ encounters binary splits at the same rate as the path to a uniform leaf, but crosses those splits differently, using the probability $\ff.$ This is formalised in \cite{PitmanWinkel2013} by the notion of \emph{bifurcator}. If $L'$ is a leaf with distribution $\mu,$ then the paths from $\rho$ to $L$ and to $L'$ will encounter countable numbers of binary branchpoints. At any common such branchpoint, let $x$ and $1-x$ be the relative $\mu$-masses of both pending subtrees, choosing $x$ to be the mass of the one containing $L$. Then the probability that $L'$ is in the other subtree, known as the ``switching probability" in \cite{PitmanWinkel2013}, is equal to $\ff(1-x)$. While the process $(Y(t),t\geq 0)$ defined by $Y(t)=\mu(\T_{[L']_{t+}})$ is well-known \cite[Section 4]{bertoin2002self} to be the Lamperti transform of a subordinator with L\'evy measure $\pi_{\mathrm{unif}},$  defined by
\begin{align*}
\int_{0}^{\infty} \mathrm{d} \pi_{\mathrm{unif}}(s) \  \phi(s) &= \sqrt{\frac{2}{\pi}} \int_{1/2}^{1} \left(x \phi(-\log (x)) + (1-x) \phi(-\log (1-x))\right)\pp(x)\dd x \\
&= \sqrt{\frac{2}{\pi}} \int_0^{1}  x \phi(-\log (x))\pp(x)\dd x,
\end{align*}
Proposition 1 from \cite{PitmanWinkel2013} states then that the process $\mu(\T_{[L]_{t+}}),t\geq 0$ will be the Lamperti transform of a subordinator where the initial $x$ in the integrand of the above is replaced by $\ff(x)$, giving us $\pi_\mu.$

That $(\nu(\T_{[L]_{t+}}),t\geq 0)$ has the appropriate paired distribution comes from the fact that its multiplicative jumps are obtained by applying $\ff$ to those of $(\mu(\T_{[L]_{t+}}),t\geq 0),$ which is how $X_{\nu}$ is derived from $X_{\mu}.$

Finally, the spinal decomposition itself is again a consequence of \cite{PitmanWinkel2013}, this time Lemma 21.
\end{proof}
\subsection{Applications: height and dimension}

Let us see two direct applications of Theorem \ref{thm:spinedec}. In the first one we compute the (annealed) law of the height of a random $\nu$-leaf, and in the second one we find the almost sure Hausdorff dimension of $\nu$.

\begin{prop}\label{prop:hd} Let $ \T$ be a Brownian CRT, and conditionally on $\T$, let $ L \in \T$ be sampled according to $\nu$. Then $ \mathrm{ht}(L)$ is a continuous random variable on $\R_+$, and its density is $8x^3\mathrm{e}^{-2x^2}.$
\end{prop}

\begin{rem} An integration by parts shows that this dominates the Rayleigh distribution ($4x\mathrm{e}^{-2x^2} \dd x,$ which is the distribution of the height of a uniform leaf) stochastically. This is expected, as the path to $L$ favours large subtrees more than the path to a uniform leaf. However, we have no natural explanation for the surprisingly simple form of this density. As one referee pointed out, this density is that of $\sqrt{\frac{G_2}{2}}$, where $G_2$ is a standard $\Gamma(2)$ random variable.
\end{rem}

\begin{proof}
We know that $\h(L)$ has the same distribution as $I=\int_{0}^{\infty}\mathrm{e}^{-\frac12 \xi_\mu(t)}\mathrm{d}t.$
The results of \cite{Carmona/Petit/Yor} show that the integer moments of this distribution can be linked to the Laplace exponent of $\xi_\mu$, which we call $\Phi$, by the following formula:

\[\mathbb{E}[(\mathrm{ht}(L))^k]=\frac{k!}{\prod_{i=1}^k \Phi(\frac{i}{2})}.\]
However, in our setting, we have for $\alpha\geq 0$
\begin{align*}\Phi(\alpha)&=\sqrt{\frac{2}{\pi}}\int_{1/2}^{1}\left(\ff(x)(1-x^{\alpha})+\ff(1-x)(1-(1-x))^{\alpha}\right)\pp(x)\dd x \\
&=\sqrt{\frac{2}{\pi}}\int_{1/2}^{1}\left(1-\ff(x)x^{\alpha}-\ff(1-x)(1-x)^{\alpha}\right)\pp(x)\dd x\\
&=\frac1{\sqrt{2\pi}}\int_{0}^{1}\left(1-\ff(x)x^{\alpha}-\ff(1-x)(1-x)^{\alpha}\right)\pp(x)\dd x.\numberthis \label{leafheightintegral}
\end{align*}

The latter integral can be computed (see Appendix \ref{sec: Appendix}), and we obtain $\Phi(\alpha)=2\sqrt{2}\alpha\frac{\Gamma(3/2+\alpha)}{\Gamma(2+\alpha)}.$ Hence
\[\mathbb{E}[(\mathrm{ht}(L))^k]=\frac{\Gamma(\frac{4+k}{2})}{2^{k/2}}.\]
A straightforward induction shows that $\frac{\Gamma(\frac{4+k}{2})}{2^{k/2}}=\int_0^{\infty} 8x^{3+k}\mathrm{e}^{-2x^2} \mathrm{d}x.$ Given that the probability distribution with density $8x^3\mathrm{e}^{-2x^2}$ is determined by its moments (since its moment generating function has positive radius of convergence), this ends the proof.
\end{proof}

\paragraph{Hausdorff dimension.}\label{section:continuous Hausdorff}

The fact that the Hausdorff dimension of the measure $\nu$ is $2 \dH$ follows from the coming proposition together with standard results about Hausdorff dimension (see e.g.~\cite[Lemma 4.1]{lyons1995ergodic} and~\cite[\S 14]{billingsley1965}):

\begin{prop}\label{prop:continuous Hausdorff} Almost surely for $\mathcal{T}$ and a.e.\ for $\nu( \mathrm{d}\ell)$ we have 
$$ \lim_{r \to 0} \frac{\log \nu \big( B_r(\ell)\big) }{\log r} = 2 \dH.$$
\end{prop}
\proof Recall that $\gamma=3(2-\sqrt{3}),$ and the definitions of $\pp(x)$ in \eqref{Stirlinglimit}, $\cc(x)$ in \eqref{eq:cc}, $\xi_\mu$ and $\xi_\nu$ in \eqref{defxi}, $X_\mu$ and $X_\nu$ in \eqref{eq:lamperti}, the time change $\theta$ in \eqref{eq:timechange} and extinction time $I$ in \eqref{def:I}. By Theorem \ref{thm:spinedec}, we can assume that the bivariate process $(\mu(\T_{[L]_{t+}}),\nu(\T_{[L]_{t+}}))_{t\geq 0}$ is given by $(X_\mu,X_\nu)$.

The local rate of growth of $\nu$ near a $\nu$-typical point is related to the asymptotic behaviour of $X_\mu$ and $X_\nu$ near their extinction time. To understand this, let us compute the first moments of the L\'evy processes $\xi_\mu$ and $\xi_\nu:$ we have 
$$ \mathbb{E}[\xi_\mu(t)] = t \int_0^{\infty} s\ \mathrm{d}\pi_\mu(s)  = \sqrt{2 \pi} \quad \mbox{and}\quad  \mathbb{E}[\xi_\nu(t)] = t \int_0^{\infty} s\ \mathrm{d}\pi_\nu(s)  = \dH \sqrt{2\pi}.$$
Since the L\'evy measure integrates the identity function, by the law of large numbers we have $t^{-1}\xi_\mu (t)  \to \sqrt{2 \pi}$   and $t^{-1}\xi_\nu (t)  \to   \dH \sqrt{2\pi}$ as $t \to \infty$ almost surely. Through the Lamperti transformation \eqref{eq:lamperti}, this turns into estimate on the time change $\theta$ since for $\veps>0$,
\[\int_{\theta(I-\veps)}^{\infty} \mathrm{e}^{-\frac12 \xi_\mu(t)}\mathrm dt=\veps, \]
from which we straightforwardly deduce using the above law of large numbers that 
\[\theta(I-\veps)\underset{\veps\to 0}\sim-\frac2{\sqrt{2\pi}}\log \veps\] almost surely. Substituting back into \eqref{eq:lamperti}, we find using the same law of large numbers that 
$$\lim_{ \varepsilon \to 0}\frac{\log X_\mu(I - \varepsilon)}{\log \varepsilon} =2\quad \mbox{and} \quad \lim_{ \varepsilon \to 0}\frac{\log X_\nu(I - \varepsilon)}{\log \varepsilon} = 2\dH,
$$ also almost surely. 

Now, by the first item of the spine decomposition (Theorem \ref{thm:spinedec}), we clearly have $\nu\big(B_{ \varepsilon }(L)\big) \leq X_\nu(I - \varepsilon)$ as a process in $ \varepsilon$ and together with the last display this entails that 
$$ \liminf_{ \varepsilon \to 0} \frac{\log \nu \big( B_\varepsilon(L)\big) }{\log  \varepsilon} \geq  2 \dH, \quad \mbox{a.s.}.$$
Let us now prove the upper bound for the dimension via a lower bound on $\nu \big( B_\varepsilon(L)\big)$. For this, we concentrate on jumps of $\xi_\mu$ of size at least $\log 2$. Remember that thanks to the spine decomposition, conditionally on such a jump of size $s\geq \log 2$ appearing at time $t$, the subtree $\T_{[L]_{u}}\setminus \T_{[L]_{u+}} \cup \{[L]_{u}\}$ with $t=\theta(u)$ (which is the subtree which is ``left on the side'' at height $u$) is an independent Brownian CRT with $\mu$-mass $$(1- \mathrm{e}^{-s}) \mathrm{e}^{- \xi_\mu(t-)} \geq \frac{1}{2} \mathrm{e}^{- \xi_\mu(t-)}.$$ Notice two things: first, the $\nu$-mass of this subtree is at least $ \mathrm{e}^{- \xi_\nu(t-)}/2$ (because $\ff(x) \leq x$ for $x \in [0,1/2]$) and second, this tree has a fixed probability equal to $ \mathbb{P}( \h(\mathcal{T}) \leq 1)$ of being of height less than the square root of its $\mu$-mass. Under those conditions, this subtree is a subset of $B_{ \varepsilon}(L)$ for $\varepsilon \geq  \mathrm{e}^{-\frac12 \xi_\mu(t-)} + \int_{t}^\infty \mathrm{e}^{-\frac12 \xi_\mu(s)} \mathrm{d}s$ and we have
\begin{eqnarray} \label{eq:lowerbound12}\nu \big(B_{ \varepsilon}(L)\big) \geq \frac12 \mathrm{e}^{- \xi_\nu(t-)}. \end{eqnarray} By the law of large numbers on $\xi_\mu$ and $\xi_\nu,$ we can take $\varepsilon = c_1 \mathrm{e}^{- t \sqrt{2\pi}/2}$  and we have  $\nu \big(B_{ \varepsilon}(L)\big) \geq c_2 \mathrm{e}^{-  t \sqrt{2\pi} \gamma/2} = \varepsilon^{2\gamma + o(1)}$ for some $c_1,c_2>0$. The preceding inequality is only valid for those $ \varepsilon$ associated to a jump time $t$ satisfying the above conditions. However, if we can find a family $t_i \to \infty$ of such times satisfying 
\begin{eqnarray} \label{eq:goaldensity} 0 < t_1 < t_2 < \cdots < t_i \xrightarrow[i \to \infty]{} \infty \quad \mbox{and} \lim_{t \to \infty} \frac{t_{i+1}}{t_i} =1, \end{eqnarray} 
then the bound $\nu \big(B_{ \varepsilon}(L)\big) \geq  \varepsilon^{2\gamma + o(1)}$ still holds as $ \varepsilon \to 0$ and proves the lower bound on the Hausdorff dimension. To see \eqref{eq:goaldensity}, notice that in the time parametrization of $\xi_\nu$, the intensity $C>0$ of such times is positive: by a standard law of large numbers we can almost surely find a diverging sequence of jump times $t_i$ satisfying all of the above assumptions and which satisfy $t_i \sim C i$. This entails \eqref{eq:goaldensity} and completes the proof.\endproof
\subsection{Scaling limits}\label{sec:scalinglimitproof}
The aim of this section is to prove Theorem \ref{thm:ghpanddim}. For the reader's convenience, we quickly recall the definition of the (pointed) Gromov--Hausdorff--Prokhorov convergence and  refer to \cite{MR3035742,khezeli2020metrization} for more details. A weighted pointed compact metric space $(M,d,\rho,\mu)$ is a compact metric space equipped with a finite Borel measure $\mu$ and a distinguished point $\rho\in M$. We let $\mathbb{M}$ be the set of all isometry classes of weighted pointed compact metric spaces. To lighten notation, we shall often identify a compact weighted metric space with its equivalence class. We equip $\mathbb{M}$ with the classical Gromov--Hausdorff--Prokhorov metric, defined for every $\textbf{M}:=(M,d, \rho,\mu)$ and  $\textbf{M}^{\prime}:=(M^{\prime},d^{{\prime}},\rho^\prime,\mu^{\prime})$ in $\mathbb{M}$ by:
$$d_{\mathrm{GHP}}\big(\textbf{M},\textbf{M}^{\prime}\big):=\inf\limits_{\phi,\phi^{\prime}}\Big( \delta_{\text{H}}\big(\phi(M),\phi^{\prime}(M^{\prime})\vee \delta_{\text{P}}\big(\phi_* \mu,\phi^{\prime}_* \mu^{\prime}\big)\vee \delta(\phi(\rho),\phi^\prime(\rho^\prime)\Big)~,$$
where the infimum is taken over all isometries $\phi$, $\phi^{\prime}$ from $M$, $M^{\prime}$ into a metric space $(Z, \delta)$ and $\delta_{\text{H}}$ (resp.~$\delta_{\text{P}}$) stands for the classical Hausdorff distance (resp.~the Prokhorov distance) in $Z$. The space $(\mathbb{M},d_{\mathrm{GHP}})$ is then a  Polish space.

Our starting point is the well-known convergence in distribution from~\cite{Ald93}:
\begin{equation}\label{eq:gh2pointed}\left(\frac{1}{\sqrt{n}}\tau_n,\rho_n,\mu_n\right) \underset{n\to\infty}{\overset{(d)}\longrightarrow} (2\sqrt{2}\T,\rho,\mu),
\end{equation}
for the pointed Gromov--Hausdorff--Prokhorov metric, where $\mu_n$ is the uniform measure on the vertices of $\tau_n$. By analogy with this notation, we shall also use the notation $\nu_n$ for the measure $\nu_{\tau_n}$ introduced in Section~\ref{section:building schemes}. We saw in Section \ref{section:building schemes}   and in Proposition \ref{prop:def:nu} that the measures $\nu_n$ and $\nu$ are obtained as measurable functions of $(\frac{1}{\sqrt{n}}\tau_n,\rho_n,\mu_n)$ and $(2\sqrt{2}\T,\rho,\mu)$ respectively. Futhermore, both measures are constructed similarly by modifying the splitting of the reference measures $\mu_n$ and $\mu$ using the functions $C(a,b)$ and $ \mathsf{c}$. In view of \eqref{eq:gh2pointed} and the convergence \eqref{eq:cc}, it is natural to expect that those functions are almost ``continuous", letting us adjoin $\nu_n$ and $\nu$ to \eqref{eq:gh2pointed} using deterministic arguments.\medskip 

\noindent \textbf{Proof of Theorem \ref{thm:ghpanddim}.} Using the Skorokhod representation theorem, we can assume that \eqref{eq:gh2pointed} holds almost surely, and the proof will proceed deterministically. Next, by a general representation theorem for Gromov-Hausdorff-type topologies~\cite[Lemma 2.5]{khezeli2023unified}, we can
embed all trees $\tau_n$ and $ \mathcal{T}$ isometrically in the same compact 
metric space $(E,d_E)$ --we keep the same notation for simplicity--, in such a way that
\begin{equation*}
\left( \frac{1}{\sqrt{n}} \tau_n,\rho_n,\mu_n\right) \underset{n\to\infty}{\longrightarrow} (2\sqrt{2}\T,\rho,\mu),
\end{equation*}
holds in pointed Hausdorff-Prokhorov sense inside $E$. Fix then any point $x$ in the skeleton of $\T$ which is not a branchpoint. Lemma \ref{lemma:annoyingmetricthing} shows that there exists a sequence $(x_n,n\in\N)$ with $x_n\in\tau_n$ such that $x_n \to x$ in $E$ and $\mu_n((\tau_n)_{x_n}) \to \mu(\T_x)$. We will now show that we also have 
\begin{equation}\label{eq:GHPlimofnu} \lim_{n \to \infty} \nu_n( (\tau_n)_{x_n})= \nu(\T_x),\end{equation}
which will guarantee via the Portmanteau theorem and \cite[Proposition 1]{Steph13}, that any subsequential limit of $(\nu_n)$ is equal to $\nu,$ ending the proof of Theorem \ref{thm:ghpanddim} since $ \mathcal{T}$ is compact.


Let us denote by $h_n$ the height of $x_n$ and by $(a^{(n)}_i,b^{(n)}_i)$ the sizes of the two subtrees above the ancestor of $x_n$ at height $0\leq i < h_n$, with the convention that $a^{(n)}_i$ is the size of the subtree containing $x_n$. It is then straightforward, but rather long, to prove that we also have the convergence 
\begin{eqnarray} \label{eq:pointlimit} \left( \left(\frac{a_i^{(n)}}{a_i^{(n)}+b_i^{(n)}} , \frac{i}{ 2\sqrt{2n}}\right) : 0 \leq i < h_n\right) \xrightarrow[]{n \to \infty} \left( \left(\frac{\mu(\T_{[x]_{t+}})}{\mu(\T_{[x]_{t}})},t\right) : 0 \leq t < \mathrm{ht}(x)\right) \end{eqnarray}
in terms of point processes on $ \mathbb{R}_+^* \times \mathbb{R}_+$. In particular, recall that \eqref{eq:noloss} holds almost surely.
Recall first from the construction of $\nu_n$ and $\mu_n$ that we have 
\begin{eqnarray} \label{rappelnun} \nu_n( (\tau_n)_{x_n}) = \prod_{i=0}^{h_n-1} C(a_i^{(n)}, b_i^{(n)})  \qquad \mbox{and} \qquad \mu_n( (\tau_n)_{x_n}) = \prod_{i=0}^{h_n-1} \frac{a_i^{(n)}}{a_i^{(n)}+b_i^{(n)}+1}. \end{eqnarray}
On the continuous side, if we set  $a_t = \mu(\T_{[x]_{t+}})$ and $b_t = \mu(\T_{[x]_{t}})-\mu(\T_{[x]_{t+}})$ for $0 \leq t \leq \mathrm{ht}(x)$ (note that $b_t=0$ when $[x]_t$ is not a branchpoint), then by \eqref{eq:noloss} and Proposition \ref{prop:def:nu} we have 
\begin{eqnarray} \label{rappelnu}\nu(x) = \prod_{t <  \mathrm{ht}(x)}  \mathsf{c}(\frac{a_t}{a_t+b_t}) \qquad \mbox{and} \qquad \mu(x) = \prod_{t < \mathrm{ht}(x)}  \frac{a_t}{a_t+b_t}. \end{eqnarray}
On the one hand, the convergences \eqref{eq:pointlimit}, \eqref{eq:cc} and Fatou's lemma directly imply that 
$$ \limsup_{n \to \infty} \nu_n( (\tau_n)_{x_n}) =  \limsup_{n \to \infty}\prod_{i=0}^{h_n-1} C(a_i^{(n)}, b_i^{(n)}) \leq  \prod_{t < \mathrm{ht}(x)}  \mathsf{c}(\frac{a_t}{a_t+b_t}) =\nu(x).$$
To get the matching lower bound $\nu(x)\leq \liminf_{n \to \infty} \nu_n( (\tau_n)_{x_n})$, notice that the same argument yields 
$$ \limsup_{n \to \infty} \mu_n( (\tau_n)_{x_n}) = \prod_{i=0}^{h_n-1} \frac{a_i^{(n)}}{a_i^{(n)}+b_i^{(n)}+1} \leq \mu(x) = \prod_{t < \mathrm{ht}(x)}  \frac{a_t}{a_t+b_t};$$ here, however, we know that the matching lower bound holds since $\lim_{n \to \infty} \mu_n( (\tau_n)_{x_n}) = \mu( \mathcal{T}_x)$ is known.  Using the fact that $C(a,b) \geq \frac{a}{a+b+1}$ for $a \geq b$, we deduce that 
$$ 1\geq \lim_{\varepsilon \to 0} \prod_{i=0}^{h_n-1} C(a_i^{(n)}, b_i^{(n)})\mathbf{1}_{b_i^{(n)} \leq \varepsilon a_i^{(n)}} \geq \lim_{\varepsilon \to 0} \prod_{i=0}^{h_n-1} \frac{a_i^{(n)}}{a_i^{(n)}+b_i^{(n)}+1}\mathbf{1}_{b_i^{(n)} \leq \varepsilon a_i^{(n)}} = 1,$$
from which \eqref{eq:GHPlimofnu} follows.

\qed

\section{Discrete fractal properties of  leaf-growth measures}
In this section, we study the discrete fractal properties of $\nu_{\tau_n}$ and focus in particular on Theorem~\ref{thm:discrete multifractal spectrum} and Theorem~\ref{thm:discretefractal}. Those results will be proved by investigating the asymptotic behaviour of the random variable $M_n=\nu_{\tau_n}(L_n)$, where $\tau_n$ is uniform in $\sT_n$ and $L_n$ is, conditionally on $\tau_n$, a random leaf of $\tau_n$ sampled according to the leaf-growth measure $\nu_{\tau_n}$. This notation will be used throughout the section.

\subsection{Typical exponent of the leaf-growth mass of a random leaf}\label{section:discrete Hausdorff}
Theorem \ref{thm:discretefractal} is a straightforward consequence of the following:
\begin{prop}
Setting $\dH=3(2-\sqrt{3})$, we have the following convergence in probability:
\[\frac{-\log M_n}{\log n} \underset{n\to\infty}{\overset{\pr}{\longrightarrow}} \dH.\]
\end{prop}
\begin{proof} For $k\in\N,$ let $L_n$ be a $\nu_{\tau_n}$-distributed random leaf of $\tau_n.$ Define two processes $\xmun(\cdot)$ and $\xnun(\cdot)$ as follows:
\begin{itemize}
    \item For $k \leq \mathrm{ht}(L_n)$, let $x$ be the ancestor of $V_n$ in $\tau_n$ with height $k$. Define $\xmun(k)$ to be $\mu_n$-mass of the subtree of $\tau_n$ rooted at $x$, and $\xnun(k)$ to be the $\nu_{\tau_n}$-mass of that subtree.
    \item For $ k> \mathrm{ht}(L_n),$ let $\xmun(k)=\xnun(k)=0$.
\end{itemize}
It is possible to show that the scaling limit of $\xmun$ is $X_\mu,$ however here we are interested in the large-time behaviour of $\xmun$ and $\xnun.$ Specifically, note then that $\h(L_n)$ is the time at which the process $X^{(n)}_{\mu}$ hits the value $\frac1{2n+1}$, and that $M_n=X^{(n)}_{\nu}(\h(L_n)).$

\paragraph{Setup.}
The proof of the proposition will be mainly based on showing weak LLN-type behaviour for both processes $X^{(n)}_{\mu}$ and $X^{(n)}_{\nu}$ at large times. Taking inspiration from \cite{BK14}, we first switch to continuous time by using a ``Poissonization" technique: let $(\mathcal{P}_n(t),t\geq 0)$ be a Poisson-like process which starts at $0$ and increases by $1$ periodically, the waiting time for the $k$-th jump being exponentially distributed with rate parameter $\sqrt{nX^{(n)}_{\mu}(k)}.$ Now set, for all $t\geq0$, $$\mathcal{X}^{(n)}_{\mu}(t)=-\log X^{(n)}_{\mu}(\mathcal{P}_n(t)).$$ This forms a Markov process: when at position $\log (2n+1) -\log (2m+1)$ for $m\geq 1$, it waits an amount of time which is exponentially distributed with parameter $m^{1/2}$ and then jumps to $\log (2n+1) -\log (2a+1),$ for $a\in\{0,\ldots,m-1\},$ with probability $2P(a,m-1-a)C(a,m-1-a).$ Notice that $\mathcal{X}^{(n)}_{\mu}$ grows from $0$ to $\log(2n+1),$ since $X^{(n)}_{\mu}$  decreased from 1 to $\frac1{2n+1}$. Similarly, for all $t\geq0$, set $\mathcal{X}^{(n)}_{\nu}(t)=-\log X^{(n)}_{\nu}(\mathcal{P}_n(t))$.

Letting $\theta_n=\inf\{t: \Xmun(t) = \log (2n+1)\}$, our aim is to show that 
\begin{equation}\label{eq:loglimittimeprob0}
\frac{ \Xnun(\theta_n)}{\log n} \underset{n\to\infty}{\overset{\pr}\longrightarrow} \dH   
\end{equation}

To do this, we will first consider $\theta'_n=\inf\{t: \Xmun(t) = \log (2n+1)-K\}$ for a large $K$ that will be chosen later. Notice that, for any value of $K$, the distribution of $\theta_n-\theta'_n$ is tight as $n$ tends to infinity, because it is a mixture of the distributions of $(\theta_i)$ for $i$ in the finite set $\{1,\ldots,\lfloor \mathrm{e}^{K}\rfloor\}.$ Similarly, the distribution of $\Xnun(\theta_n)-\Xnun(\theta'_n)$ is also tight, as it is a mixture of the distributions of $\mathcal{X}^{(i)}(\theta_i)$ for $i\in \{1,\ldots,\lfloor \mathrm{e}^{K}\rfloor\}$. This implies that we only need to prove Equation \eqref{eq:loglimittimeprob0} with $\theta'_n$ replacing $\theta_n$:
\begin{equation}\label{eq:loglimittimeprob}
\frac{ \Xnun(\theta'_n)}{\log n} \underset{n\to\infty}{\overset{\pr}\longrightarrow} \dH.  
\end{equation}
We set out to prove \eqref{eq:loglimittimeprob}.

\paragraph{Analysis of $\Xmun$ and estimate of $\theta'_n$.} Here we show that $\Xmun(t)$ is mostly linear in $t$, and deduce that $\theta'_n$ is well approximated by $\frac{2\log n}{\sqrt{\pi}}.
$

Since $\mathcal{X}^{(n)}_{\mu}$ is a pure jump process, we can compute its predictable compensator. It is given by integrating the local drift $D_\mu^{(n)}(t)=\frac{1}{ \mathrm{d}t} \mathbb{E}[\mathcal{X}^{(n)}_{\mu}(t+\dd t)-\mathcal{X}^{(n)}_{\mu}(t)\mid \mathcal{F}_t]$ which is  equal to the rate of jumps, times the expectation of the next jump, that is
\[ D_{\mu}^{(n)}(t)=m^{1/2}\sum_{a=0}^{m-1} -\log\left(\frac{2a+1}{2m+1}\right) 2P(a,m-1-a)C(a,m-1-a), \]
with $m$ such that $\mathcal{X}^{(n)}_{\mu}(t)=\log n - \log m.$
We can apply Proposition \ref{prop:generalintegrallimit}, part $(ii)$, to this sum (since $\log (1/m)=o(\sqrt{m})$ and $x\log(x)$ tends to $0$ as $x$ tends to $0$), obtaining that it converges as $m$ tends to infinity to
\[\frac1{2\sqrt{\pi}}\int_0^1 -\log x \, \ff(x) \pp(x)\dd x=\frac{\sqrt{\pi}}{2}.\]
The fact that this is a limit as $m$ tends to infinity is useful because the drift will be close to this value at any $t$, as long as $m$ is big (i.e., as long as $\mathcal{X}_n(t)< \log(n)-K$ for large $K$). By definition of the predictable compensator, the process $(M^{(n)}_{\mu}(t),t\geq 0)$ defined by 
\[M^{(n)}_{\mu}(t)=\Xmun(t)-\int_0^t D_{\mu}(t)\mathrm dt,\]
is a $(\mathcal{F}_t : t \geq 0)$-martingale. Its quadratic variation is given by integrating the rate of jumps, times the expectation of the square of the next jump, which is
\[V_{\mu}^{(n)}(t)=m^{1/2}\sum_{a=0}^{m-1} \left(\log\left(\frac{2a+1}{2m+1}\right)\right)^2 2P(a,m-1-a)C(a,m-1-a),\]
where $m$ is still defined by  $\mathcal{X}^{(n)}_{\mu}(t)=\log n - \log m.$ Again by Proposition \ref{prop:generalintegrallimit}, this converges as $m \to \infty$ to
\[\frac1{2\sqrt{\pi}}\int_0^1 (\log x)^2\ff(x)\pp(x)\dd x.\]
In particular, this means that that $B:=\underset{n,t}\sup \, V^{(n)}_{\mu}(t)<\infty.$

\vspace{1cm}

Since $M^{(n)}_{\mu}(0)=0$, we have
\[\E[M^{(n)}_{\mu}(t)^2]=\E\left[\int_{0}^t V_{\mu}(t)\mathrm d t\right]\leq Bt,\]
implying via Doob's inequality that
\[\pr[\underset{u\leq t} \sup \, |M^{(n)}_{\mu}(u)|>t^{0.6}]\leq \frac{Bt}{t^{1.2}} \underset{t\to\infty}\longrightarrow 0, \]
uniformly in $n$.

Let now $\veps>0,$ and choose $K$ large enough to have $|D_{\mu}(t)-\frac{\sqrt{\pi}}{2}|\leq \frac{\veps}{2}$ whenever $\Xmun(t)\leq \log n - K.$ Consider $t_n=\frac{2(1-\veps)\log n}{\sqrt{\pi}}$ and let $E_n$ be the event that $\underset{t\leq t_n}\sup \,|M^{(n)}_{\mu}(t) |\leq t_n^{0.6}$. We claim that, for $n$ large enough, on $E_n$, we have $\theta'_n\geq t_n.$ To see this, notice that, on $E_n$, we have for all $t\leq t_n$
\[\Xmun(t)=\int_0^t D^{(n)}_{\mu}(s)\mathrm ds + O(\log(n)^{0.6}) \]
However, if $\theta'_n\leq t_n$, then for $n$ large enough we would have
\begin{align*}
\log(n)-K &\leq \Xmun(\theta'_n)\\
&\leq \frac{\sqrt{\pi}}{2}(1+\frac{\veps}{2})\theta'_n +O(\log(n)^{0.6}) \\
   &\leq (1+\frac{\veps}{2})(1-\veps)\log n +O(\log(n)^{0.6}),
\end{align*}
which eventually gives a contradiction. Put together, this means that $\pr[\theta'_n\geq \frac{(1-\veps)\log n}{A}]\geq \pr[E_n] \to 1.$

Let now $u_n=\frac{2(1+\veps)\log n}{\sqrt{\pi}},$ and this time let $F_n$ be the event that $\underset{t\leq u_n}\sup| M^{(n)}_{\mu}(t)| \leq u_n^{0.6},$ whose probability also tends to $1$. Let us now check that, on $F_n$, $\theta'_n\leq u_n.$ Similarly to before, observe that, on $F_n$, we have for all $t\leq u_n$ that 
\[\Xmun(t)=\int_0^t D^{(n)}_{\mu}(s)\mathrm ds + O(\log(n)^{0.6}). \]
Thus, if we had $\Xmun(u_n)< \log(n)-K,$
then we would have $V^{(n)}_{\mu}(s)\geq A(1-\frac{\veps}{2})$ for all $s\leq u_n,$ and in particular
\[
\log n \geq \Xmun(u_n)\geq \frac{\sqrt{\pi}}{2}(1-\frac{\veps}{2})\frac{2(1+\veps)\log n}{\sqrt{\pi}}+ O((\log n)^{0.6}),\]
a contradiction for $n$ large enough. Hence $\pr[\theta'_n\leq \frac{2(1+\veps)\log n}{\sqrt{\pi}}]\geq \pr[F_n]\rightarrow 1.$

\paragraph{Analysis of $\Xnun$ and conclusion.}
We can use the same method to establish the first order linear behaviour of $\Xnun$. Keeping the notation $\Xmun(t)=\log (2n+1) - \log (2m+1)$, we have that $\Xnun$ jumps at rate $m^{1/2},$ and if $\Xmun$ jumps to $\log (2n+1) - \log (2a+1)$, then the value of $\Xnun$ increases by $\log(C(a,b))$. Hence
\[ D_{\nu}(t)=m^{1/2}\sum_{a=0}^{m-1} \log (C(a,b)) 2P(a,b)C(a,b), \]
the limit of which is
\[\frac1{2\sqrt{\pi}}\int_{0}^1 \ff(x)^2\pp(x)\dd x=\frac{\sqrt{\pi}}{2}\gamma.\]
Similarly, the local variance is then
\[ V_{\nu}^{(n)}(t)=m^{1/2}\sum_{a=0}^{m-1} \log(C(a,b))^2 2P(a,b)C(a,b), \]
the limit of which is
\[\frac1{2\sqrt{\pi}}\int_{0}^1 \log(\ff(x))^2\ff(x)\pp(x)\dd x,\]
hence we have $V_{\nu}(t)\leq B$ for all $t$, up to changing $B.$

We consider the martingale $(M^{(n)}_{\nu}(t),t\geq 0)$ defined by 
\[M^{(n)}_{\nu}(t)=\Xmun(t)-\int_0^t D_{\nu}(t)\mathrm dt.\]
The same arguments as for $M^{(n)}_{\nu}$ show that, if $G_n$ is the event that $\underset{t\leq u_n}\sup |M^{(n)}_{\nu}(t)| \leq u_n^{0.6},$ then $\pr[G_n]$ tends to $1$. Up to increasing $K$, we can moreover assume that $|D_{\nu}(t)-\frac{\sqrt{\pi}}{2}\gamma|\leq \veps$ for $t\leq \theta'_n.$ Then on the event $G_n\cap\{|\theta'_n-\frac{2}{\sqrt{\pi}}\log(n)|\leq \frac{2}{\sqrt{\pi}}\veps\log(n)\},$ which has probability tending to $1$, we have
\[|\Xnun(\theta'_n)-\frac{\sqrt{\pi}}{2}\gamma\theta'_n|\leq \frac{\veps}{2}\theta'_n+O((\log n)^{0.6}),\]
from which \eqref{eq:loglimittimeprob} readily follows.
\end{proof}

\subsection{Moments of the leaf-growth mass of a random leaf: the multifractal spectrum}\label{section:multifractal spectrum}
In contrast to the result of the previous section, we now highlight the multifractality phenomenon governing the behaviour of the random variable $M_n$ and proceed to compute the full discrete multifractal spectrum of the measure $\nu_{\tau_n}$.

Our aim within this section will be to estimate all moments $\E[M_n^\alpha]$ for $\alpha\in\R$; by the end, we shall prove Theorem~\ref{thm:discrete multifractal spectrum}. Before we do, let us give a brief insight into the appearance of the integral $I(\alpha,\beta)$ from \eqref{eq:I(alpha,beta)}; we shall then show that, given $\alpha\in\R$, there is a unique $\beta(\alpha)$ such that $I(\alpha,\beta(\alpha))=0$. Finally, we dive into a formal proof of Theorem~\ref{thm:discrete multifractal spectrum}.

In order to compute $\E[M_n^\alpha]$, we can disintegrate according to the profile of $\tau_n$; indeed, letting $\tau_n^L$ and $\tau_n^R$ be the left and right subtrees, conditionally on $|\tau_n^L|=a$, the leaf $L_n$ is distributed according to $\nu_{\tau_n^L}$ among the $a+1$ leaves of $\tau_n^L$ with probability $C(a,b)$ (where $b=n-1-a$), and distributed according to $\nu_{\tau_n^R}$  among the $b$ leaves of $\tau_n^R$ with probability $1-C(a,b)$. In the first case, the mass $M_n$ has the distribution of $C(a,b)M_a$, and in the second case it is distributed as $C(b,a)M_{b}$. This yields
\[\E[M_n^\alpha]=\sum_{a=1}^{n-1}P(a,b)\left(C(a,b)^{\alpha+1}\E[M_a^\alpha]+C(b,a)^{\alpha+1}\E[M_b^\alpha]\right).\]
Since we have $\sum_{a=1}^{n-1}P(a,b)=1$, we can rewrite this as
\[\sum_{a=0}^{n-1}P(a,b)\left(C(a,b)^{\alpha+1} \E[M_a^\alpha]+C(b,a)^{\alpha+1} \E[M_b^\alpha]-\E[M_n^\alpha]\right)=0.\]

Now, if we informally assume an asymptotic of the form $\E[M_n^\alpha] \approx n^{-\beta}$, where $\beta$ is some positive real number, and recall that $C(a,b)$ is close to $\ff(\frac{a}{n}),$ we obtain that
\[\sum_{a=1}^{n-2}P(a,b)\left(\ff(a/n)^{\alpha+1} (a/n)^{-\beta}+\left(1-\ff(a/n)\right)^{\alpha+1} \left(b/n\right)^{-\beta}-1\right)\]
should be close to $0$. However, similarly to Proposition \ref{prop:generalintegrallimit}, we expect to this sum converge to $\frac{1}{4\sqrt{\pi}} I(\alpha,\beta),$ hence if $\E[M_n^\alpha] \approx n^{-\beta}$ then we should have $I(\alpha,\beta)=0$.

Let us now consider the problem of determining $\beta(\alpha)$ as a function of $\alpha$:
\begin{lem}
For $\alpha\in\R$, the quantity
\[I(\alpha,\beta)=\int_0^1 \left(\ff(x)^{\alpha+1} x^{-\beta}+\ff(1-x)^{\alpha+1} (1-x)^{-\beta}-1\right)\pp(x)\dd x\]
is finite for all $\beta \in (-\infty,2\alpha+\frac{3}{2}).$ It is continuous and strictly increasing as a function of $\beta$, with limits
$$\lim_{\beta\to -\infty}I(\alpha,\beta)=-\infty\mbox{ and }\lim_{\beta\to2\alpha+\frac{3}{2}^-}I(\alpha,\beta)=+\infty.$$
\end{lem}

\begin{proof}
By symmetry, we only need to check what happens when $x$ is close to $0$. Since $\ff(1)=1$ and $\ff$ is differentiable, we have $\ff(1-x)^{\alpha+1} (1-x)^{-\beta}-1 \sim Ax$ for some $A\in\R$, hence $(x(1-x))^{-3/2}(\ff(1-x)^{\alpha+1} (1-x)^{-\beta}-1)$ is integrable, while since $\ff(x)\sim x^2,$ $(x(1-x))^{-3/2}\ff(x)^{\alpha+1}x^{-\beta}\sim x^{2\alpha-\beta+1/2}$ is integrable only if $\beta<2\alpha+\frac{3}{2}.$

Monotonicity, continuity and limit properties are direct consequences of the same properties for the integrand and the monotone convergence theorem.
\end{proof}

It follows that, for each $\alpha\in\R$, there is a unique $\beta(\alpha)$ such that $I(\alpha,\beta(\alpha))=0$.

\begin{proof}[Proof of Theorem~\ref{thm:discrete multifractal spectrum}]Fix $\alpha\in\R$, let $\beta=\beta(\alpha)$ as well as $e_n=\E[M_n^\alpha]$. We will prove that, for all $\veps>0,$ $e_n=O(n^{-\beta+\veps})$ and $e_n=\Omega(n^{-\beta-\veps}),$ from which the result follows straightforwardly.

We start with the upper bound. For $n\in\N,$ let $C_n=\underset{0\leq k\leq n}\sup \,e_k (k+1)^{\beta-\veps};$ we aim to show by induction that $C_n$ stays bounded as $n$ tends to infinity. Recall the equation
\[e_n=\sum_{a=0}^{n-1}P(a,b)\left(C(a,b)^{\alpha+1}e_a+C(b,a)^{\alpha+1}e_b\right)\]
where $b$ is implicitly always equal to $n-1-a.$ We then have
\[e_n\leq C_{n-1} \sum_{a=0}^{n-1}P(a,b)\left(C(a,b)^{\alpha+1}(a+1)^{-\beta+\veps}+C(b,a)^{\alpha+1}(b+1)^{-\beta+\veps}\right).\]
Using the fact that $\sum_{a=0}^{n-1} P(a,b)=1,$ we can now write
\[\frac{e_n}{C_{n-1}(n+1)^{-\beta+\veps}}\leq 1+ \sum_{a=0}^{n-1}P(a,b)\left(C(a,b)^{\alpha+1}\left(\frac{a+1}{n+1}\right)^{-\beta+\veps}+C(b,a)^{\alpha+1}\left(\frac{b+1}{n+1}\right)^{-\beta+\veps}-1\right).\]
While the sum on the RHS is not of the type featured in Proposition \ref{prop:generalintegrallimit}, it is reasonable to expect that
\begin{equation}\label{eq:annoyingintegrallimit}
    \lim_{n\to\infty} \sqrt{n}\sum_{a=0}^{n-1}P(a,b)\left(C(a,b)^{\alpha+1}\left(\frac{a+1}{n+1}\right)^{-\beta+\veps}+C(b,a)^{\alpha+1}\left(\frac{b+1}{n+1}\right)^{-\beta+\veps}-1\right)=\frac1{4\sqrt{\pi}}I(\alpha,\beta-\veps).
\end{equation}
We postpone checking \eqref{eq:annoyingintegrallimit} to the end of this proof.

Now, noting that $I(\alpha,\beta-\veps)<0,$ we then get that, for $n$ large enough,
\[\frac{e_n}{C_{n-1}n^{-\beta+\veps}}\leq 1-\frac{A}{\sqrt{n}} \]
for some $A>0,$ hence
$e_n\leq C_{n-1}(1-\frac{A}{\sqrt{n}})(n+1)^{-\beta+\veps}\leq C_{n-1}(n+1)^{-\beta+\veps}.$
This shows that $C_{n+1}=C_n,$ and so the sequence $(C_n,n\in\N)$ is eventually constant.

\vspace{0.5cm}

The lower bound is proved the same way: letting $c_n=\underset{k\leq n}\inf \,e_k (k+1)^{\beta+\veps},$ we obtain
\[\frac{e_n}{c_{n-1}(n+1)^{-\beta-\veps}} \geq 1 + \frac1{4\sqrt{\pi n}}I(\alpha,\beta+\veps) + o(\frac1{\sqrt{n}}),\]
and since $I(\alpha,\beta+\veps)>0$ this implies that the sequence $(c_n,n\in\N)$ is eventually constant, completing our proof.

\vspace{0.5cm}
We now prove statement \eqref{eq:annoyingintegrallimit}. We will use bounds \eqref{Stirlingbound} and \eqref{eq:boundc1} as well as
\[\frac1{3} \cc{\left(\frac{a}{n}\right)}\leq C(a,b)\leq 3 \cc{\left(\frac{a}{n}\right)},\]
for $a+b=n-1.$ As with Proposition \ref{prop:generalintegrallimit}, it is sufficient to only consider
\[\sqrt{n}\sum_{a=2}^{\lfloor n/2 \rfloor}P(a,b)\left(C(a,b)^{\alpha+1}\left(\frac{a+1}{n+1}\right)^{-\beta+\veps}+C(b,a)^{\alpha+1}\left(\frac{b+1}{n+1}\right)^{-\beta+\veps}-1\right).\]
This sum is equal to $\int_0^{1/2} g_n(x)\mathrm{d}x,$ where 
\[g_n(x)= n^{3/2}P(a,b)\left(C(a,b)^{\alpha+1}\left(\frac{a+1}{n+1}\right)^{-\beta+\veps}+C(b,a)^{\alpha+1}\left(\frac{b+1}{n+1}\right)^{-\beta+\veps}-1\right)\] if $x\in[\frac{a-1}{n},\frac{a}{n})$ with $a\in \{2,\ldots,\lfloor n/2 \rfloor\},$ and $g_n(x)=0$ otherwise. It is readily checked that $g_n(x)$ converges pointwise towards $\frac{1}{4\sqrt{\pi}} \pp(x)\left(\ff(x)^{\alpha+1} x^{-\beta+\veps}+\ff(1-x)^{\alpha+1} (1-x)^{-\beta+\veps}-1\right),$
and we need to check that the convergence is dominated. Notice the following, where the constant $K>0$ may vary from line to line:
\begin{itemize}
    \item Since $n^{3/2} P(a,b)\left(\frac{a}{n}\right)^{3/2}\left(\frac{b}{n}\right)^{3/2}$ is uniformly bounded in $a$ and $n$, and $x<\frac{1}{2},$ we have $n^{3/2}P(a,b)\leq K x^{-3/2}.$
    \item Since $\frac1{3}\ff(\frac{a}{n})\leq C(a,b)\leq 3\ff(\frac{a}{n})$ and $x\leq \frac{a}{n}\leq 2x,$ we have $\frac1{3}\ff(x)\leq C(a,b)\leq 3\ff(2x).$ Since moreover $x\leq\frac{a+1}{n+1}\leq 3x,$ we deduce that $C(a,b)^{\alpha+1} \left(\frac{a+1}{n+1}\right)^{-\beta+\epsilon}\leq K x^{2(\alpha+1)}x^{-\beta+\veps},$ independently of the signs of $\alpha+1$ and $\beta.$
    \item Similarly, we have $|1-C(b,a)|\leq Kx^2,$ and $\left|1-\frac{b+1}{n+1}\right|\leq Kx$ hence $\left|C(b,a)^{\alpha+1}\left(\frac{b+1}{n+1}\right)^{-\beta+\epsilon}-1\right| \leq K x.$
\end{itemize}

 Combining these together yields a constant K such that
 \[|g_n(x)|\leq \frac{K}x^{-3/2}\left(x^{2\alpha+2-\beta+\veps}+x\right),\]
 which is a uniform integrable bound on $(0,1/2).$
\end{proof}

\begin{rem}\label{rem:multi-fractal}It is easy to prove by induction on $n \geq 1$ that the smallest possible leaf-growth weight of a leaf in a binary tree of size $n$ is given by $C(0,n-1) \sim \frac{3}{2 n^2}$, which is the weight of a leaf attached to the root of such a tree. An interesting question raised by a referee is whether one could show that there exists a constant $s_{\max}$ such that 
$$ \max_{l \in \tau_n} \nu_{\tau_n}(l) = n^{-s_{\max} + o_{\mathbb{P}}(1)},$$ where $o_{\mathbb{P}}(1)$ is a function which tends to $0$ in probability. More generally, an alternative way to present the multi-fractal spectrum is to prove the existence of a function $\zeta$ such that 
$$ \nu_{\tau_n}(l \in \tau_n \mid \nu_{\tau_n}(l) = n^{-s+o(1)}) = n^{-\zeta(s)+o_{\mathbb{P}}(1)}.$$ Given the expression above, we would expect (but do not prove here) that $\zeta$ and $\beta$ should be related by 
$$ \beta(\alpha) = \inf_{s \geq 0} \big( s\cdot (\alpha+1) + \zeta(s)\big).$$
\end{rem}

\section{Towards a diffusion limit}
In this work we studied in details  the leaf-growth measure on large binary trees and their continuous limit. This is the first step towards understanding the possible  Markov process  on real trees that one would obtain by passing the discrete Luczak--Winkler growth procedure to the limit. In the following we heuristically call ``CRT dynamic" any Markov process with values in compact real trees with invariant measure given by the Brownian CRT. 

\bigskip

\paragraph{Stationary dynamics.} A natural way to define a CRT dynamic is to take a scaling limit of a Markov chain $(T^{(n)}_k : k \geq 1)$ on (say) the space of binary trees with fixed size $n$ with uniform  invariant distribution. There are many such chains: the flip dynamic \cite{mcshine1997mixing}, Aldous' move on cladograms \cite{lohr2020aldous,forman2023aldous}, subtree pruning and re-graft  \cite{EW06}... Passing those dynamics to the scaling limit as $n \to \infty$ requires first to understand the appropriate scaling in time which is related to the mixing time of the chain: we want to find $f(n)$ so that 
$$ \left( T^{(n)}_{[f(n)t]} , t \geq 0\right)$$ converges towards a CRT dynamic. When the Markov chain is itself reversible, the dynamic can even be extended to $ \mathbb{R}$. Finding $f(n)$ is usually a hard problem and has been the subject of many investigations in recent years. 

\paragraph{Growing dynamics.} Another route is to consider a growing chain  $(T_n , n \geq 1)$ on (say) the set of binary trees so that $T_n$ is uniformly distributed on plane binary trees of size $n$ for each $n \geq 1$ fixed. Those chains do not require to understand the mixing time since the time and size are intimately tight and the limit CRT-dynamic should be the scaling limit as $n \to \infty$ of
 \begin{eqnarray} \label{eq:dynamiCRT}  \left( \frac{T_{[n t]}}{ \sqrt{n t}},t\geq 0\right).\end{eqnarray}

The most well-known growth procedure is perhaps Rémy's algorithm \cite{Rem85}, where iteratively, a uniform edge of $T_n$ is split in its middle to yield $T_{n+1}$. When the choices of edges are made independently for each  $n$, the  resulting Markov chain converges almost surely towards a Brownian CRT \cite{CurienHaas}, so that the limiting CRT dynamic obtained in \eqref{eq:dynamiCRT} is constant (which is not so interesting). However, still in the R\'emy growth, there is a less well-known way of coupling the choices of edges, due to Bacher, Bodini and Jacquot \cite{bacher2017efficient}, which yields a different Markov chain, whose growth is more ``local". The scaling limits of this chain given in \eqref{eq:dynamiCRT} is under current investigation and is a CRT dynamic with ``local and continuous growth". 

The Luczak--Winkler leaf-growth mechanism studied in these pages is totally different from Rémy's algorithm. We conjecture that if $(T^{lw}_n, n \geq 1)$ is a sequence of uniform binary trees  so that conditionally on $ (T^{lw}_k, 1 \leq k \leq n)$, the tree $ T^{lw}_{n+1}$ is obtained by growing a cherry on point sample according to the leaf-growth measure on $ T^{lw}_n$,  then \eqref{eq:dynamiCRT} yields a non-trivial CRT-dynamic $( \mathcal{LW}_t, t\geq 0)$. This dynamic is introduced in \cite{CFT25} as a special case of a more general growing mechanism for self-similar Markov trees. The following proposition indicates that the limiting CRT dynamic is mixing:

\begin{prop}\label{prop:pairedindep}
Consider a sequence $(m_n,n\in\N)$ which tends to $\infty$ such that $m_n=o(n)$. We abuse notation and skip the $n$ index in the sequel. We have
\[ \left(\frac{1}{\sqrt{m}}T^{lw}_m,\frac{1}{\sqrt{n}}T^{lw}_n\right) \underset{n\to\infty}{\overset{(d)}\longrightarrow} (2\sqrt{2}\,\T,2\sqrt{2}\,\T'),\]
where $\T'$ is an independent copy of $\T$.
\end{prop}

\begin{proof} By construction, $T^{lw}_n$ can be seen as a tree obtained from  $T^{lw}_m$ by grafting trees at each of the $m+1$ leaves of $T^{lw}_m.$ More precisely, if we denote by $S_1(n,m),... , S_{m+1}(n,m)$ the sizes of the subtrees $( \theta_1(n,m), ... , \theta_{m+1}(n,m))$ grafted on the leaves of $T^{lw}_m$ to get $T^{lw}_n$, it should be clear from the dynamics that conditionally on the vector $(S_1(n,m), ... , S_{m+1}(n,m))$ and on $T^{lw}_m$, the random trees $(\theta_1(n,m), ... , \theta_{m+1}(n,m))$ are independent and uniform binary trees of sizes $(S_1(n,m), ... , S_{m+1}(n,m))$. Since the random tree $T^{lw}_m$ has diameter of order $ \sqrt{m} = o_{\mathbb{P}}(\sqrt{n})$ the root of the trees $\theta_1(n,m), ... , \theta_{m+1}(n,m)$ are almost confounded in $T^{lw}_n$. Since the root of the Brownian CRT is almost surely a point of degree $1$, it follows that as $n,m \to \infty$ satisfying $ \frac{n}{m} \to \infty$, then with high probability there exists a unique index $1 \leq I(n,m) \leq m+1$ so that $$n^{-1} \cdot S_{I(n,m)}(n,m) \to 1 \quad \mbox{and} \quad  n^{-1/2}\cdot \max_{i \ne I(n,m)} \mathrm{Height}(\theta_i(n,m)) \to 0,$$ in probability. In particular, the geometry of $ n^{-1/2}\cdot T^{lw}_n$ is close to that of $ S_{I(n,m)}(n,m)^{-1/2}\cdot \theta_{I(n,m)}(n,m)$ which converges towards a Brownian CRT, while being independent of $T^{lw}_m$, and hence the respective scaling limits are also independent. \end{proof} 

Let us conclude this paper with the following remark: as in \cite{bacher2017efficient} for the Rémy growth, it seems possible in our case to \textit{couple} the sampling of the leaves in the leaf-growth mechanism to get even more ``local" growth process $(\tilde{T}_n^{lw} : n \geq 0)$ whose dynamical scaling limits seem easier to establish than that $({T}_n^{lw} : n \geq 0)$. We hope to address those questions in future work.

\section*{Appendix A: computing integral \eqref{leafheightintegral}}\label{sec: Appendix}
We are going to prove that, for all real $\alpha\geq 0$,
\begin{equation}\label{eq:Phi(alpha)}
\Phi(\alpha) = \frac{1}{\sqrt{2\pi}}\int_{0}^1 \left( 1-x^{\alpha}\ff(x)-(1-x)^{\alpha}\ff(1-x) \right) \pp(x) \dd x=2\sqrt{2}\alpha\frac{\Gamma(\frac{3}{2}+\alpha)}{\Gamma(2+\alpha)}.
\end{equation}
Recall that for $x, y \in \mathbb{C}$ with $\Re x, \Re y > 0$ the Beta function is defined by
\[
B(x,y) = \int_0^1 t^{x-1} (1-t)^{y-1} \dd t.
\]
and satisfies the fundamental equality 
\[
B(x,y) = \frac{\Gamma(x)\Gamma(y)}{\Gamma(x+y)}.
\]

Note that, since $\Gamma(1/2)=\sqrt{\pi}$, what we are claiming is that $\sqrt{\frac{\pi}{2}}\Phi(\alpha)=2\alpha B( \frac{3}{2}+\alpha, \frac{1}{2} )$. In order to show this, one can use the following lemma:

\begin{lem}\label{lemma:MysteriousPiece}
For $\alpha\in\mathbb{C}$ such that $\Re \alpha > \frac{1}{2}$, we have
\[
\int_0^1 (1-x^\alpha-(1-x)^\alpha) (x(1-x))^{-3/2} \dd x = 4\alpha \left( B\left(\alpha+\frac{1}{2}, \frac{1}{2}\right)  -B\left(\frac{3}{2}, \alpha-\frac{1}{2}\right) \right).
\]
\end{lem}
\begin{proof}
    Let $F(x)=1-x^\alpha-(1-x)^\alpha$, $g(x)=(x(1-x))^{-3/2}$ and
    \[
        G(x) = \frac{2(2x-1)}{\sqrt{x(1-x)}}.
    \]

One can immediately check that $G(x)$ is an antiderivative of $g(x)$; 
    thus, integrating by parts and noticing that the boundary term $F(x)G(x)$ has limit equal to $0$ both as $x \to 0$ and $x \to 1$, we find that the integral in the statement is equal to
    \[
    \begin{aligned}
    -\int_0^1 F'(x) G(x) \dd x & =
\int_0^1 \left(\alpha x^{\alpha-1}-\alpha (1-x)^{\alpha-1}\right) \frac{2(2x-1)}{\sqrt{x(1-x)}} \dd x  \\
& = 2\alpha \int_0^1 \left( 2x^{\alpha} -x^{\alpha-1}-2x(1-x)^{\alpha-1} + (1-x)^{\alpha-1} \right) \frac{1}{\sqrt{x(1-x)}} \dd x \\
& = 2\alpha \left( 2B\left(\alpha+\frac{1}{2}, \frac{1}{2}\right) - B\left(\alpha-\frac{1}{2}, \frac{1}{2}\right) -2B\left(\frac{3}{2}, \alpha-\frac{1}{2}\right)+B\left(\frac{1}{2},\alpha-\frac{1}{2}\right) \right) \\
& = 2\alpha \left( 2B\left(\alpha+\frac{1}{2}, \frac{1}{2}\right)  -2B\left(\frac{3}{2}, \alpha-\frac{1}{2}\right) \right),
\end{aligned}
    \]
    where we have used the symmetry $B\left(\frac{1}{2},\alpha-\frac{1}{2}\right)=B\left(\alpha-\frac{1}{2}, \frac{1}{2}\right)$.
\end{proof}
Now, in order to compute \eqref{eq:Phi(alpha)}, we can just express $f_3(x)$ as $2x^2(1-x) + x^2$, thus decomposing $2\sqrt{\frac{\pi}{2}}\Phi(\alpha)$ as the sum of three integrals $I_1(\alpha), I_2(\alpha), I_3(\alpha)$:
\[
I_1(\alpha) := \int_0^1 -x^{\alpha} 2x^2 (1-x) x^{-3/2}(1-x)^{-3/2} \dd x = -2B\left(\frac{3}{2}+\alpha, \frac{1}2\right).
\]
\[
I_2(\alpha) := \int_0^1 -(1-x)^{\alpha} 2(1-x)^2 x x^{-3/2}(1-x)^{-3/2} \dd x = -2B\left(\frac{1}2, \frac{3}{2}+\alpha\right).
\]
\[
I_3(\alpha) = \int_0^1 (1-x^{\alpha+2}-(1-x)^{\alpha+2}) (x(1-x))^{-3/2} \dd x.
\]
Using Lemma \ref{lemma:MysteriousPiece} with parameter $\alpha+2$ (thus merely requiring the condition that $\Re \alpha > -3/2$), the quantity $I_3(\alpha)$ can be rewritten as
\[
I_3(\alpha) = 2(2\alpha+4) \left( B\left( \frac{5}{2}+\alpha, \frac{1}{2} \right)  -B\left(\frac{3}{2}, \frac{3}{2}+\alpha\right) \right).
\]
On the other hand, one has
\[
B\left( \frac{5}{2}+\alpha, \frac{1}{2} \right) = \frac{\Gamma\left( \frac{5}{2}+\alpha\right) \Gamma\left(\frac{1}{2} \right)}{\Gamma\left(3+\alpha\right)} = \frac{(\frac{3}{2}+\alpha) \Gamma\left( \frac{3}{2}+\alpha\right) \Gamma\left(\frac{1}{2} \right)}{(2+\alpha) \Gamma\left(2+\alpha\right)} = \frac{2\alpha+3}{2\alpha+4} B\left(\frac{3}{2}+\alpha, \frac{1}{2} \right)
\]
and
\[
B\left( \frac{3}{2}, \frac{3}{2}+\alpha \right) = \frac{\Gamma\left(\frac{3}{2} \right) \Gamma\left(\frac{3}{2}+\alpha \right)}{\Gamma\left( 3+\alpha \right)} =  \frac{\frac{1}{2}\Gamma\left(\frac{1}{2} \right) \Gamma\left(\frac{3}{2}+\alpha \right)}{(2+\alpha)\Gamma\left( 2+\alpha \right)} = \frac{1}{2\alpha+4} B\left(\frac{3}{2}+\alpha, \frac{1}{2} \right).
\]
Setting for simplicity $B:=B\left(\frac{3}{2}+\alpha, \frac{1}{2} \right)$, we have obtained
\[
2\sqrt{\frac{\pi}{2}}\Phi(\alpha) = I_1(\alpha) + I_2(\alpha) + I_3(\alpha) = -4B + 2(2\alpha+4)\left( \frac{2\alpha+3}{2\alpha+4}B-\frac{1}{2\alpha+4} B \right)=-4B + 2(2\alpha+2)B =4\alpha B,
\]
as desired.

\section*{Appendix B: limits of trees and subtrees}
\begin{lem}\label{lemma:annoyingmetricthing}
Consider a sequence $((\T_n,\rho_n,\mu_n),n\in\N)$ of deterministic rooted $\R$-trees which are subsets of a fixed compact metric space $(E,d),$ such that $\T_n\to \T,$ $\rho_n\to \rho$ and $\mu_n\to\mu$ for respectively the Hausdorff metric $d_{\mathrm{H}}$ on compact subsets of $E$, the metric $d$ on $E$, and the Prokhorov metric $d_{\mathrm{P}}$ for Borel probability measures on $E$.
Then, for any $z\in\T$, there exists a sequence $(z_n,n\in\N)$ with $z_n\in T_n$ such that $z_n\to z$ and $\mu_n((\T_n)_{z_n}) \to \mu(\T_z)$.
\end{lem}

\begin{proof}
Write
\[\veps_n := 2  \max(d_{\mathrm{H}}(\T_n,\T) , d(\rho_n,\rho), d_{\mathrm{P}}(\mu_n,\mu)),\] this tends to $0$ as $n$ tends to infinity.

 Let $h$ be the height of $z$, for all $n$, $x_n$ an element of $\T_n$ such that $d(z_n,x_n)\leq \veps_n,$ and $y_n$ be the ancestor of $x_n$ in $\T$ with height $\mathrm{ht}(x_n)-\veps_n$,  Notice that $d(y_n,z)\leq 3\veps_n$

Let us show that $(\T_n)_{y_n}$ converges to $T_z$ for the Hausdorff metric. We use the notation $A^{\veps}$ to denote the $\veps$-enlargement of a subset $A$ of $E$ and will show that we have both $\T_z\subset (\T_n)_{y_n}^{\veps_n}$ and $(\T_n)_{y_n}\subset \T_z^{\veps'_n}$ with $\veps'_n\to 0.$ 

Let $a\in \T_z$, we know that there is $b\in \T_n$ such that $d(a,b)\leq \veps_n,$ and we will show that $b$ is in $(\T_n)_{y_n}.$ By the triangle inequality we have $d(b,x_n)\leq d(b,a)+d(a,z)+d(z,x_n)\leq d(a,z)+2\veps_n$ and $d(a,z)\leq d(b,x_n)+2\veps_n,$ hence $|d(b,x_n)-d(a,z)|\leq x\veps'_n.$ Similarly we obtain $|d(b,\rho_n)-d(a,\rho)|\leq 2\veps_n$ and $|d(x_n,\rho_n)-d(z,\rho)|\leq 3\veps_n$  Since $d(a,\rho)+d(a,z)+d(z,\rho)$, we deduce that $d(b,x_n)+d(x_n,\rho_n)-d(b,\rho_n)\leq 6\veps_n.$ Since $\T_n$ is an $\R$-tree, we can see that this difference is equal to twice the the distance from $x_n$ to its most recent common ancestor with $b$ is at most $3\veps_n$, hence this most recent common ancestor is a descendant of $y_n$, hence $b\in (\T_n)_{y_n},$ hence $\T_z\subset (\T_n)_{y_n}^{\veps_n}.$

Similarly, omitting the details, it can be seen that $(\T_n)_{y_n}\subset \T_{u_n},$ where $u_n$ is an ancestor of $z$ such that $d(u_n,z)\to 0.$ However, since $\T$ is compact, we deduce that $\T_{u_n}\subset \T_z^{\veps'_n},$ where $\veps'_n$ tends to $0$, otherwise there would be a sequence with no subsequential limits.

We can now finish the proof. Letting $\eta_n=\max(\veps_n,d_{\mathrm{H}}((\T_n)_{y_n},T_z))$ for all $n$, we have $\T_z\subset (\T_n)^{\veps_n}_{z'_n},$ we now finally define $z_n$ as the ancestor of $z'_n$ such that $\mathrm{ht}(z_n)=\mathrm{ht}(z'_n)-\eta_n.$ The same argument as in the previous paragraph shows that $(\T_n)_{z_n}$ converges to $\T_z$ for the Hausdorff distance, hence $\eta'_n:=\max(\eta_n,d_{\mathrm{H}}((\T_n)_{z_n},T_z))$ tends to $0$, and we have the following inclusions of sets:
\[\T_z\subset (\T_n)_{z'_n}^{\eta_n} \subset (\T_n)_{z_n} \subset \T_z^{\eta'_n}.\] By the definition of the Prokhorov metric on measures, we then have
\[\mu(\T_z)\leq \mu_n((\T_n)_{z_n})+\eta_n\leq \mu(\T_z^{\eta'_n})+\eta_n+\eta'_n.\]
However, since $\T_z$ is a closed set, $\mu(\T_z^{\eta'_n})$ converges to $\mu(\T_z)$, and by the sandwich theorem, so does $\mu_n((\T_n)_{z_n})$.
\end{proof}
\bibliographystyle{alpha}
\bibliography{frag}
\end{document}